\documentclass[reqno,oneside,12pt]{amsart}

\usepackage{amsmath,amsfonts,amssymb}
\usepackage{tikz}
\usepackage{tikz-cd}
\tikzset{node distance=1.5cm, auto}
\usepackage{color}
\definecolor{darkgreen}{rgb}{0,0.45,0}
\usepackage[pagebackref,colorlinks,citecolor=darkgreen,linkcolor=darkgreen]{hyperref}
\usepackage[capitalize]{cleveref}
\usepackage{stmaryrd}
\usepackage{enumerate}
\usepackage[shortlabels]{enumitem}
\usepackage{mathtools}

\usepackage{graphicx}
\usepackage{enumerate}
\usepackage[all]{xy}
\usepackage[latin9]{inputenc}

\newcommand{\End}{{\sf End}}

\newcommand{\Rep}{{\sf Rep}}

\newcommand{\PRep}{{\sf PRep}}

\newcommand{\x}{\otimes}

\def\eps{\epsilon}

\newtheorem{prop}{Proposition}[section]
\newtheorem{proposition}[prop]{Proposition}
\newtheorem{lemma}[prop]{Lemma} 
 
\newtheorem{corollary}[prop]{Corollary}

\newtheorem{theorem}[prop]{Theorem}

\theoremstyle{definition}

\newtheorem{defi}[prop]{Definition}
\newtheorem{definition}[prop]{Definition}
\newtheorem*{notation}{Notation}

\newtheorem{example}[prop]{Example}

\newtheorem{remark}[prop]{Remark}

\newcommand{\benu}{\begin{enumerate}}
\newcommand{\enu}{\end{enumerate}}

\newcommand{\beqna}{\begin{eqnarray}}
\newcommand{\eqna}{\end{eqnarray}}
\newcommand{\beqnast}{\begin{eqnarray*}}
\newcommand{\eqnast}{\end{eqnarray*}}
\newcommand{\beqn}{\begin{equation}}
\newcommand{\eqn}{\end{equation}}
\newcommand{\beqnst}{\begin{equation*}}
\newcommand{\eqnst}{\end{equation*}}

\usepackage{latexsym}

\usepackage{xspace}
\usepackage{amscd}
\usepackage{amssymb}
\usepackage{amsfonts}

\makeatother

\newcommand{\bema}{\left ( \begin{array}}
\newcommand{\ema}{\end{array} \right )}

\begin{document}
	
	\title[Partial representations of connected and smash product Hopf algebras]{Partial representations of connected and \\ smash product Hopf algebras}

	\author[T.\ Ferrazza]{Tiago Luiz Ferrazza}
	\address{\textup{(Ferrazza)} Departamento de Matem\'atica, Universidade Federal do Paran{\'a}, Brazil}
    \email{tiago.ferrazza@unespar.edu.br}

    \author[W.\ Hautekiet]{William Hautekiet}
	\address{\textup{(Hautekiet)} D\'epartement de Math\'ematiques, Universit\'e Libre de Bruxelles, Belgium}
	\email{william.hautekiet@ulb.be}

 \author[A.\ Neto]{Arthur Alves Neto}
	\address{\textup{(Neto)} Departamento de Matem\'atica, Universidade Federal do Paran{\'a}, Brazil}
    \email{arthurnt@ufpr.br}
	
	
	\flushbottom
	
	\begin{abstract} 
		We show that every partial representation of a connected Hopf algebra is global. Some interesting classes of partial representations of smash product Hopf algebras are studied, and a description of the partial ``Hopf'' algebra if the first tensorand is connected is given. If \(H\) is cocommutative and has finitely many grouplikes, this allows to see \(H_{par}\) as the weak Hopf algebra coming from a Hopf category.
	\end{abstract}
	
	\maketitle
	
	\tableofcontents
	
	\section*{Introduction}
	\thispagestyle{empty} 
	
	Over the last two decades, partial representations of groups and Hopf algebras have been intensively studied. Partial actions of groups originated in the theory of \(C^*\)-algebras \cite{exel}, describing those \(C^*\)-algebras that have an action of the circle group, as generalized crossed products using partial automorphisms. The notion of partial representation of groups was introduced in \cite{exel2} and it was shown in \cite{DEP} that the partial representations of a finite group \(G\) correspond to representations of a groupoid \(\Gamma(G)\). This result was generalized in \cite{alves2015partial}, where partial representations of Hopf algebras were studied. The category of partial representations of a Hopf algebra \(H\) is isomorphic to the category of usual representations of a suitably constructed Hopf algebroid \(H_{par}\). One interesting feature of partial representations is that they are not only based on the algebraic properties of the Hopf algebra, but also on the coalgebraic properties. This indicates that studying partial representations can probe deeper in the structure of a Hopf algebra.

	However, until now, partial representations have only been explicitly described for 3 classes of Hopf algebras, which we list here in increasing order of complexity of their partial representation theory:
	\begin{enumerate}
	\item Universal enveloping algebras of Lie algebras. These Hopf algebras do not admit partiality \cite[Example 4.4]{alves2015partial}, i.\,e. every partial representation is a global (usual) representation and the partial ``Hopf'' algebra \(H_{par}\) is just the Hopf algebra itself.
	\item Finite group algebras. In this case the partial representations correspond to representations of a groupoid \cite{DEP}, hence the partial group algebra \(k_{par} G\) is a weak Hopf algebra.
	\item Sweedler's 4-dimensional Hopf algebra (and more generally, the Hopf-Ore extensions of \(kC_2\), as studied in \cite[Chapter 4]{arthurthesis}). The base algebra \(A_{par}\) of the Hopf algebroid \(H_{par}\) becomes infinite-dimensional in this case, which implies that the Hopf algebroid is not associated to a weak Hopf algebra.
	\end{enumerate}
	
	The aim of this paper is to develop tools to describe the partial representations for more general classes of Hopf algebras. First, we show that there exists a larger class of Hopf algebras that do not possess partiality: the connected Hopf algebras. More precisely, we show that if a partial representation of \(H\) is multiplicative on the coradical \(H_0\), then it is in fact a global representation. If \(H\) is connected (i.\,e. \(H_0\) is trivial), this shows that any partial representation of \(H\) is global. 
	This way we obtain new examples of Hopf algebras that do not have partiality, thus broadly extending the example of universal enveloping algebras. Conversely, we show that any Hopf algebra that has a nontrivial cosemisimple Hopf quotient admits at least one partial representation which is not global. This hypothesis is satisfied by large classes of Hopf algebras, such as cocommutative Hopf algebras in characteristic 0 which are not connected. Indeed, recall that if $H$ is a cocommutative Hopf algebra over a field of characteristic $0$, then by the famous Cartier-Gabriel-Konstanz-Milnor-Moore theorem (see for instance \cite[Theorem 5.6.5]{montgomery}), 
 $H$ is a smash product of a universal enveloping of a Lie algebra \(\mathfrak{g}_H\) and a group algebra \(kG_H\). From this one can easily deduce that the coradical of \(H\) is exactly \(kG_H\), and that the projection \(\pi : H \to kG_H\) is a Hopf algebra map. Hence there are always nontrivial partial representations if the group is nontrivial.
	This again indicates that the coalgebraic properties (such as the coradical) play an important role in understanding the partial representations of Hopf algebras.
	
	We aim to go a step further here, and fully describe the partial representations of cocommutative Hopf algebras (over an algebraically closed field of characteristic zero). This is what we will do in the second part of the paper, but we approach the problem in a more general setting. Following \cite{CIMZsmash}, we consider two Hopf algebras $U$ and $H$, together with a smash product map $R : H \otimes U \to U \otimes H$ turning $U\#_R H$ into a Hopf algebra (to have this, it suffices that \(R\) satisfies just three conditions; it has to be \textit{normal, multiplicative} and a \textit{coalgebra map}). Examples of such smash product Hopf algebras are obtained from actions of one Hopf algebra on the other, and by exact factorizations of groups. In \cref{se:main}, we study partial representations \(\pi\) of $U\#_R H$ that split as a product of a partial representation of \(U\) and a partial representation of \(H\), i.\,e.
	\begin{align*}
		\pi(u \x h) &= \pi(u \x 1_H) \pi(1_U \x h), \\
		\pi(R(h \x u)) &= \pi(1_U \x h)\pi(u \x 1_H), 
	\end{align*}
	for all \(u \in U, h \in H\). We show that these are equivalent to representations of a certain smash product of the partial ``Hopf'' algebras \(U_{par}\) and \(H_{par}\), with a smash product map \(\mathcal{R}\) induced by \(R\).
	
	The partial representations of \(U \#_R H\) that restrict to a global representation of \(U\) form an interesting subcategory of these; they remind of the partial representations of a group \(G\) that are global on a given subgroup as studied in \cite{DHSV}. We show that they are equivalent to the category of representations of \(U\#_{\mathcal{T}} H_{par}\), where \(\mathcal{T} : H_{par} \x U \to U \x H_{par}\) is again obtained from \(R\). This is of particular interest if \(U\) is a Hopf algebra without partiality, and allows to describe the partial ``Hopf'' algebra of any cocommutative Hopf algebra as a smash product: since \(U(\mathfrak{g})\) is connected, all of its partial representations are global, so any partial representation of \(U(\mathfrak{g}) \#_R kG\) restricts to a global representation of \(U(\mathfrak{g})\). We obtain that
	\[(U(\mathfrak{g}) \#_R kG)_{par} \cong U(\mathfrak{g}) \#_{\mathcal{T}} k_{par} G.\]
	If the number of grouplikes is finite, then \(k_{par} G\) is isomorphic to a groupoid algebra, and 
 \(U(\mathfrak{g}) \#_{\mathcal{T}} k_{par} G\) is in fact a weak Hopf algebra coming from a certain Hopf category. 
	
	This article is organized as follows. In \cref{se:prelim}, we recall the necessary notions and properties of partial representations and the coradical. Our first main result, that any partial representation of a connected Hopf algebra is global, is shown in \cref{se:connected}, and we give a sufficient condition for Hopf algebras to possess at least one partial representations which is not global. In \cref{se:smash}, we first recall the definition of a smash product algebra, and give the sufficient conditions for it to be a Hopf algebra. We also show some new technical properties there. Then we study partial representations of smash product Hopf algebras, with the description of \(H_{par}\) for cocommutative \(H\) as a main application. Some further examples to illustrate the results are given.

	\begin{notation}
		Throughout the article, \(H\) is a Hopf algebra over a field \(k\). For the comultiplication, we adopt the Sweedler notation
		\[\Delta(h) = h_{(1)} \otimes h_{(2)}.\]
		For vector spaces \(V\) and \(W,\) the tensor flip \(V \otimes W \to W \otimes V\) is denoted by \(\tau_{V, W}\).
	\end{notation}
	
	\section{Preliminaries}
	\label{se:prelim}
	
	In this section we will recall the basic notions about partial representations of Hopf algebras and the coradical filtration.
	
	\subsection{Partial representations of Hopf algebras}
	
	\begin{definition}[\cite{alves2015partial}]
		A partial representation of \(H\) on an algebra \(B\) is a linear map \(\pi : H \to B\) such that for all \(h, k \in H\)
		\begin{enumerate}[(PR1)]
			\item \(\pi(1_H) = 1_B\); \label{PR1}
			\item \(\pi(h) \pi(k_{(1)}) \pi(S(k_{(2)})) = \pi(hk_{(1)}) \pi(S(k_{(2)}))\); \label{PR2}
			\item \(\pi(h_{(1)}) \pi(S(h_{(2)})) \pi(k) = \pi(h_{(1)}) \pi(S(h_{(2)})k)\). \label{PR3}
		\end{enumerate}
		By \cite[Lemma 3.3]{ABCQVparcorep}, axioms \ref{PR2} and \ref{PR3} can be replaced with
		\begin{enumerate}[(PR1)]
			\setcounter{enumi}{3}
			\item \(\pi(h) \pi(S(k_{(1)})) \pi(k_{(2)}) = \pi(hS(k_{(1)})) \pi(k_{(2)})\); \label{PR4}
			\item \(\pi(S(h_{(1)})) \pi(h_{(2)}) \pi(k) = \pi(S(h_{(1)}))\pi(h_{(2)}k)\). \label{PR5}
		\end{enumerate}
	\end{definition}
	If \(B = \End_k(M)\) for some \(k\)-vector space \(M,\) then we call \(M\) a left partial \(H\)-module. 
	From the axioms it follows immediately that an algebra morphism (i.\,e.\ a usual representation) \(H \to B\) is a partial representation. We will often refer to those as \textit{global representations} of \(H\) on \(B\). 
	
	It turns out that the partial representations of \(H\) factor uniquely through the partial ``Hopf'' algebra \(H_{par}\). This algebra is constructed as the quotient of the tensor algebra \(T(H)\) by the relations
	\begin{gather}
		1_H = 1_{T(H)}; \label{eq:Hparrelation1} \\
		h \otimes k_{(1)} \otimes S(k_{(2)}) = hk_{(1)} \otimes S(k_{(2)}); \label{eq:Hparrelation2}\\
		h_{(1)} \otimes S(h_{(2)}) \otimes k = h_{(1)} \otimes S(h_{(2)})k; \label{eq:Hparrelation3}
	\end{gather}
	for all \(h, k \in H\). The class of \(h \in H\) in \(H_{par}\) is commonly denoted as \([h]\).
	It is easy to see that the linear map 
	\begin{equation}
		\label{eq:bracket}
		[-] : H \to H_{par} : h \mapsto [h]
	\end{equation}
	is a partial representation. In the other direction there is an algebra map
	\begin{equation}
		\label{eq:removing_brackets}
		H_{par} \to H : [h_1] \cdots [h_n] \mapsto h_1 \cdots h_n.
	\end{equation}
	
	The algebra \(H_{par}\) satisfies the following universal property.
	\begin{theorem}[{\cite[Theorem 4.2]{alves2015partial}}]
		\label{th:universal_property}
		For every partial representation \(\pi : H \to B\) there is a unique algebra morphism \(\hat{\pi} : H_{par} \to B\) such that \(\pi = \hat{\pi} \circ [-]\). Conversely, given an algebra morphism \(\varphi : H_{par} \to B,\) the map \(\varphi \circ [-] : H \to B\) is a partial representation.
	\end{theorem}
	This tells us that there is a bijective correspondence between partial representations of \(H\) and representations of \(H_{par}\), and that the category of left partial \(H\)-modules is isomorphic to the category of left \(H_{par}\)-modules.
	
	The structure of \(H_{par}\) can best be described with respect to a particular subalgebra \(A_{par}\), which is generated by the elements of the form
	\begin{equation}
		\label{eq:epsilonh}
		\varepsilon_h = [h_{(1)}] [S(h_{(2)})] \qquad \text{for } h \in H.
	\end{equation}
	As was shown in \cite[Section 4.4]{alves2015partial}, this algebra can also be constructed as a quotient of the tensor algebra \(T(H)\). Its defining relations are
	\begin{gather}
		1_{A_{par}} = \varepsilon_{1_H}; \label{eq:Arelation1} \\
		\varepsilon_h = \varepsilon_{h_{(1)}}  \varepsilon_{h_{(2)}}; \label{eq:Arelation2} \\
		\varepsilon_{h_{(1)}} \varepsilon_{h_{(2)}k} = \varepsilon_{h_{(1)}k} \varepsilon_{h_{(2)}}; \label{eq:Arelation3}
	\end{gather}
	for all \(h, k \in H\).

\begin{definition}[\cite{alves2015partial}] A \textit{left partial action} of a Hopf algebra $H$ on an algebra $A$ is a linear map \(H \otimes A \to A\), given by $h \otimes a \mapsto h \cdot a$, such that
\begin{itemize}
    \item[(PA$1$)]\label{A} $1_H\cdot a = a$,
    \item[(PA$2$)] $h\cdot (ab) = (h_{(1)} \cdot a)(h_{(2)} \cdot b)$,
    \item[(PA$3$)] $h\cdot (k\cdot a) = (h_{(1)} \cdot 1_A)(h_{(2)} k\cdot a)$,
\end{itemize}
for all $a, b\in A$ and $h,k \in H$. The algebra $A$ is called a \textit{partial left $H$-module algebra}. A left partial action is \textit{symmetric} if in addition, it satisfies
\begin{itemize}
    \item[(PA$4$)] $h\cdot (k\cdot a) = (h_{(1)} k\cdot a)(h_{(2)} \cdot 1_A)$, 
\end{itemize}
for all $h,k \in H$ and $a\in A$.

     Let $A$ be a partial $H$-module algebra. Consider the associative multiplication
     $$ (a\x h)(b\x k) = a(h_{(1)}\cdot b)\x h_{(2)} k $$
     on \(A \otimes H\). Then, the space $A\underline{\#} H = (A\x H)(1_A\x 1_H)$ is a unital algebra, which is generated by elements of the form
$$ a\# h = a(h_{(1)} \cdot 1_A)\x h_{(2)} $$
for \(a \in A\) and \(h \in H\).
This algebra is called the partial smash product of $A$ by $H$ \cite{CJpartialaction}.
\end{definition}
 
The following theorem summarizes the main results obtained in \cite{alves2015partial} about the structure of \(H_{par}\) with respect to its subalgebra \(A_{par}\).
	
	\begin{theorem}[{\cite[Theorems 4.8 and 4.10]{alves2015partial}}]
		\begin{enumerate}[(i)]
			\item There is a partial action of \(H\) on \(A_{par}\) defined by
			\begin{equation*}
				h \cdot a := [h_{(1)}] a [S(h_{(2)}]
			\end{equation*}
			for \(h \in H, a \in A\).
			\item The partial ``Hopf'' algebra \(H_{par}\) is isomorphic to the partial smash product algebra \(A_{par} \underline{\#} H\). The partial representation \eqref{eq:bracket} becomes
			\begin{equation}
				\label{eq:bracket_smash}
				H \to A_{par} \underline{\#} H : h \mapsto 1_{A_{par}} \# h.
			\end{equation}
        \item If \(H\) has invertible antipode, then $H_{par}$ has the structure of a Hopf algebroid over the subalgebra $A_{par}$.
		\end{enumerate}
	\end{theorem}
	
	Let us look at some examples to illustrate how partial representations of different Hopf algebras can behave very differently. 
	
	\begin{example}
		\label{ex:C2}
		Let \(G = C_2 = \{1, g\}\) be the cyclic group of order two and \(H = kC_2\) its group algebra. From relation \eqref{eq:Hparrelation2} for \(h = k = g\) it follows that \([g]^3 = [g]\). Hence \(kC_{2, par}\) is 3-dimensional and has as a basis \(\{[1], [g], [g]^2\}\). Apart from the global representations, there is up to isomorphism just one simple partial representation, which is described by
		\[kC_2 \to k : \begin{cases}
			1 \mapsto 1, \\
			g \mapsto 0.
		\end{cases}\]
	\end{example}
	
	\begin{example}
		Consider \(H = kC_2^*,\) the Hopf algebra dual to the one from the previous example. If \(\mathrm{char}(k) \neq 2,\) then \(kC_2^* \cong kC_2\) as Hopf algebras, so their partial representation theory is the same. If \(\mathrm{char}(k) = 2\) however, then \(kC_2^*\) and \(kC_2\) are not isomorphic. Let \(\{p_1, p_g\}\) be the basis of \(kC_2\) dual to \(\{1, g\}\). Then
		\begin{align*}
			\Delta(p_1) &= p_1 \otimes p_1 + p_g \otimes p_g, \\
			\Delta(p_g) &= p_1 \otimes p_g + p_g \otimes p_1.
		\end{align*}
		Combining relations \eqref{eq:Arelation2} and \eqref{eq:Arelation3} one can show that \(\varepsilon_{p_g} = 0,\) so \([p_1] [p_g] = [p_g][p_1]\). Using \eqref{eq:Hparrelation2} for \(h = k = p_1,\) we find
		\[[p_1]^2 = [p_1]^3 + [p_1] [p_g]^2 = [p_1] ([p_1]^2 + [p_g]^2) = [p_1] ([p_1] + [p_g])^2 = [p_1]\]
		because \(p_1 + p_g = 1_{kC_2^*}\). In a similar way, \([p_g]^2 = [p_g]\) and \([p_1][p_g] = [p_g][p_1] = 0\). This shows that \(kC^*_{2, par} = kC_2^*\) and that every partial representation of \(kC_2^*\) is global if \(\mathrm{char}(k) = 2\). 
	\end{example}
	
	\begin{example}
		\label{ex:lie}
		Let \(\mathfrak{g}\) be a Lie algebra and \(U(\mathfrak{g})\) its universal enveloping algebra. Every partial representation of \(U(\mathfrak{g})\) is global, as is shown in \cite[Example 4.4]{alves2015partial}.
	\end{example}
	
	\begin{example}
		\label{ex:sweedler}
		Let \(H\) be Sweedler's 4-dimensional Hopf algebra over a field of characteristic different of 2 (see for instance \cite[Example 1.5.6]{montgomery}). 
		In \cite[Example 4.13]{alves2015partial}, a description of \(H_{par}\) was given, and in particular it was shown that the base algebra \(A_{par}\) of the Hopf algebroid \(H_{par}\) is infinite-dimensional and isomorphic to \(k[x, y]/(2x^2 - x, 2xz - z)\). This shows that \(H_{par}\) does not have the structure of a weak Hopf algebra, since those have a separable Frobenius (in particular finite-dimensional) base algebra. In \cite{arthurthesis}, \(H_{par}\) was described in another way:
		\[H_{par} \cong k[x, y]/(xy) \times H\]
		as algebras.
	\end{example}
	
	\subsection{The coradical filtration}

  In the next section we will show that for a connected Hopf algebra, any of its partial representations is in fact global. The proof crucially uses the so-called coradical filtration.
	
	\begin{definition}
		The coradical $H_0$ of \(H\) is the sum of all simple subcoalgebras of $H$. 
		For \(n \geq 1\), recursively define
		$$ H_n=\Delta^{-1}(H\x H_{n-1}+H_0\x H).$$
		The chain 
		\[H_0 \subseteq H_1 \subseteq \cdots \subseteq H_n \subseteq \cdots\]
		is called the coradical filtration of \(H\).
	\end{definition}

	\begin{lemma}[\cite{montgomery}]
		\label{le:coradical_filtration}
		\begin{enumerate}[(i)]
			\item For each \(n \geq 0\), \(H_n\) is a subcoalgebra of \(H\).
			\item The coradical filtration is a coalgebra filtration, i.\,e.\
			\[\Delta(H_n) \subseteq \sum_{i = 0}^n H_i \otimes H_{n - i}  \text{ for all } n, \quad \text{ and } \quad H = \bigcup_{n\geq 0}H_n. \]
		\end{enumerate}
	\end{lemma}
	
	\begin{definition}
		\begin{enumerate}[(i)]
			\item A Hopf algebra is \textit{pointed} if every simple subcoalgebra is one-dimensional, i.~e.~if the coradical is isomorphic to a group algebra.
			\item A Hopf algebra is \textit{connected} if \(H_0 = k1_H\).
		\end{enumerate}
	\end{definition}
	
	It is clear that any connected Hopf algebra is pointed. Group algebras are always pointed (since for those \(H = H_0 = kG\)). The universal enveloping algebras of Lie algebras from \cref{ex:lie} are connected. The Sweedler Hopf algebra is pointed but not connected: its coradical is \(\langle 1, g \rangle \cong kC_2\) and \(H_1 = H\).

\section{Measuring partiality by the coradical}
	\label{se:connected}
	
	In \cref{ex:lie} it was observed that universal enveloping algebras of Lie algebra do not possess any partiality, i.~e.~every partial representation of such a Hopf algebra is global. In this section we show that this is true for any connected Hopf algebra, and we formulate conditions to ensure the existence of partial representations.
	
	\subsection{Connected Hopf algebras have no partiality}
	
	\begin{lemma}
		\label{le:global_partial}
		Let $\pi:H \to B$ be a partial representation. Then \(\pi\) is a global representation (i.\,e.\ an algebra morphism) if and only if for all \(h \in H\) 
		\begin{equation}
			\label{eq:eps_trivial}
			\pi(h_{(1)})\pi(S(h_{(2)})) = \eps(h)1_B.
		\end{equation}
	\end{lemma}
	\begin{proof}
		The direct implication is trivial. For the converse, remark that it follows from \ref{PR3} that for all \(h \in H\)
		\[\pi(h_{(1)}) \pi(S(h_{(2)})) \pi(h_{(3)}) = \pi(h).\]
		So if \eqref{eq:eps_trivial} holds, then
		\begin{align*}
			\pi(h) \pi(k) &\stackrel{\mathmakebox[\widthof{  =  }]{}}{=} \pi(h_{(1)}) \pi(S(h_{(2)})) \pi(h_{(3)}) \pi(k) \\
			&\stackrel{\mathmakebox[\widthof{  =  }]{\ref{PR5}}}{=} \pi(h_{(1)}) \pi(S(h_{(2)})) \pi(h_{(3)}k) \\
			&\stackrel{\mathmakebox[\widthof{  =  }]{\eqref{eq:eps_trivial}}}{=} \pi(hk)
		\end{align*}
		for all \(h, k \in H\).
	\end{proof}
	
	Given a partial representation \(\pi : H \to B\) and \(h \in H\), we will write
	\begin{equation}
		\varepsilon_{h}^{\pi} \coloneqq \pi(h_{(1)})\pi(S(h_{(2)})) \in B.
	\end{equation} 
	
	The following theorem is an easy application of \cref{th:universal_property}. 
	\begin{theorem}
		Let \(H\) be a Hopf algebra. The following are equivalent:
		\begin{enumerate}[(i)]
			\item Every partial representation of \(H\) is global.
			\item The map \([-] : H \to H_{par}\) is an algebra map.
			\item The map \([-] : H \to H_{par}\) is an isomorphism of algebras.
			\item \(A_{par} = k 1_{H_{par}}\).
		\end{enumerate}
	\end{theorem}
	\begin{proof}
		The equivalence of the first three items follows from \cref{th:universal_property} and the description of \(H_{par}\) as a quotient of \(T(H)\).  For the equivalence with the last point, recall that \(A_{par}\) is generated by the elements \(\varepsilon_h = \varepsilon_h^{[-]}\) \eqref{eq:epsilonh}. If \(\varepsilon_h \in k1_{H_{par}}\), then \(\varepsilon_h = \epsilon(h) 1_{H_{par}}\) (which can be seen by applying the algebra map \eqref{eq:removing_brackets}). Now the result follows from applying \cref{le:global_partial} to the partial representation \([-] : H \to H_{par}\).
	\end{proof}
	
	\begin{lemma}\label{le:trivial_epsilon_VW}
		
		Let $\pi:H \to B$ be a partial representation and let $V$ and $W$ be two linear subspaces of $H$ such that for all \(h \in V \cup W\),
		$$ \varepsilon_{h}^{\pi} = \epsilon(h)1_B. $$
		Then $\varepsilon_h^{\pi} = \eps(h)1_B$, for all $h \in \Delta^{-1}(H\x V + W\x H)$.
		
	\end{lemma}
	
	\begin{proof}
		Take $h \in \Delta^{-1}(H\x V + W\x H)$. Then there exist finite families $x_i, y_j \in H$, $v_i \in V$ and  $w_j \in W $ such that
		\begin{equation} \Delta(h) =  \sum_{i} x_i\x v_i + \sum_{j}w_j\x y_j. \label{eq:hVW1} \end{equation}
		Applying \(\epsilon\) to either tensorand of the equality yields
		\begin{equation}
			h =  \sum_{i} \eps(x_i)v_i + \sum_{j} \eps(w_j)y_j = \sum_{i} \eps(v_i)x_i + \sum_{j} \eps(y_j)w_j. \label{eq:hVW2} 
        \end{equation}
        This implies that
		\begin{equation}
			\label{eq:epsh2}
			\varepsilon^{\pi}_{h} = \sum_{i} \eps(x_i)\eps(v_i)1_B + \sum_{j} \eps(w_j)\varepsilon^{\pi}_{y_j} = \sum_{i} \eps(v_i)\varepsilon^{\pi}_{x_i} + \sum_{j} \eps(y_j)\eps(w_j)1_B
		\end{equation}
    because \(\varepsilon^\pi_h\) is linear in \(h\).
    On the other hand, it follows from \ref{PR2} that
		$$ \pi(h_{(1)})\pi(S(h_{(2)})) = \pi(h_{(1)})\pi(S(h_{(2)}))\pi(h_{(3)})\pi(S(h_{(4)})),$$
		so $\varepsilon^{\pi}_h = \varepsilon^{\pi}_{h_{(1)}}\varepsilon^{\pi}_{h_{(2)}}$. Combining this with (\ref{eq:hVW1}) gives
		\begin{equation} 
			\label{eq:epsh1}
			\varepsilon^{\pi}_h = \sum_{i} \varepsilon^{\pi}_{x_i}\varepsilon^{\pi}_{v_i} + \sum_{j}\varepsilon^{\pi}_{w_j}\varepsilon^{\pi}_{y_j} = \sum_{i} \epsilon(v_i)\varepsilon^{\pi}_{x_i} + \sum_{j}\epsilon(w_j)\varepsilon^{\pi}_{y_j} 
		\end{equation}
		because \(v_i, w_j \in V \cup W\).
		Finally,
		\begin{align*}
			\varepsilon^{\pi}_h & \stackrel{\mathmakebox[\widthof{   =   }]{}}{=}  \varepsilon^{\pi}_h + \varepsilon^{\pi}_h - \varepsilon^{\pi}_h\\
			& \stackrel{\mathmakebox[\widthof{   =   }]{(\ref{eq:epsh2},\, \ref{eq:epsh1})}}{=} \Bigl(\sum_{i} \eps(x_i)\eps(v_i)1_B + \sum_{j} \eps(w_j)\varepsilon^{\pi}_{y_j}\Bigr) + \Bigl(\sum_{i} \eps(v_i)\varepsilon^{\pi}_{x_i} + \sum_{j} \eps(y_j)\eps(w_j)1_B\Bigr) \\
			& \stackrel{\mathmakebox[\widthof{   =   }]{}}{-} \Bigl(\sum_{i} \eps(v_i)\varepsilon^{\pi}_{x_i} + \sum_{j}\eps(w_j)\varepsilon^{\pi}_{y_j}\Bigr) \\
			& \stackrel{\mathmakebox[\widthof{   =   }]{}}{=} \Bigl(\sum_{i} \eps(x_i)\eps(v_i) + \sum_{j} \eps(y_j)\eps(w_j)\Bigr)1_B \\
			&\stackrel{\mathmakebox[\widthof{   =   }]{}}{=}\eps(h)1_B. \qedhere
		\end{align*}
	\end{proof}
	
	This lemma becomes very useful when applying it to the coradical filtration of \(H\).
	
	\begin{theorem}
		\label{th:coradicalglobal}
		Let $H$ be a Hopf algebra, $H_0$ the coradical of $H$, and $\pi:H \to B$ a partial representation of $H$ on $B$. Then $\pi:H \to B$ is a global representation if and only if
		\[\varepsilon^{\pi}_h = \eps(h)1_B \qquad  \text{for all }h \in H_0.\]
	\end{theorem}
	\begin{proof}
		Apply \cref{le:trivial_epsilon_VW} inductively, taking \(V = H_{n - 1}\) and \(W = H_0\). Since \(\bigcup_{n \geq 0} H_n= H,\) the result follows from \cref{le:global_partial}.
	\end{proof}
	
	Recall that a Hopf algebra is called \textit{connected} when its coradical is 1-dimensional. We get the following corollary from \cref{th:coradicalglobal}.
	\begin{corollary}
		\label{cor:connected}
		Let \(H\) be a connected Hopf algebra. Then every partial representation is global, i.\,e.\ \(H_{par} = H\). 
	\end{corollary}
	
	\begin{example}
	\label{ex:connected}
	We already encountered a first class of connected Hopf algebras in \cref{ex:lie}: universal enveloping algebras of Lie algebras. For those, the conclusion of \cref{cor:connected} was already known from \cite{alves2015partial}. In \cite{class_connected}, a classification of the connected Hopf algebras of Gelfand-Kirillov dimension at most four is given, if the field is algebraically closed and of characteristic zero. In \cite{gilmartin}, an example of Gelfand-Kirillov dimension 5 is constructed. This connected Hopf algebra is not isomorphic as an algebra to any universal enveloping algebra of a Lie algebra. For the interested reader, we describe two more classes of examples of connected Hopf algebras.
	\begin{enumerate}[(i)]
		\item Combinatorial Hopf algebras. These typically have a basis indexed by a certain combinatorial object, e.\,g. the Connes-Kreimer algebra \(\mathcal{H}_R\) which is generated by rooted forests. Combinatorial Hopf algebras are graded by the size of the object (the number of vertices of the rooted forest), and they are always connected, since there is only one object of size zero (the empty tree). 
		
		\noindent The product is the linear extension of some \textit{combination rule}, describing the ways two objects can be put together, while the coproduct is given by some \textit{decomposition rule}, describing how an object can be broken up. For instance, in the Connes-Kreimer algebra, the product is given by disjoint union of rooted forests, and the coproduct is found by (admissibly) cutting the trees in all possible ways. We refer to \cite{ConnesKreimer} for more details. The Connes-Kreimer algebra is not cocommutative because the piece of the cut that contains the root is always placed in the second tensorand.
		
		\noindent Another example of a combinatorial Hopf algebra is the shuffle algebra \(\mathcal{S}(N)\), introduced in \cite{ree}. It has as a basis the set of words in the alphabet \(\{1, \dots, N\}\). The product of two words is given by the sum of all their  interleavings, and the coproduct by deconcatenation.

		\item Dual group algebras in positive characteristic. Let \(k\) be a field of characteristic \(p\) and \(G\) be a finite \(p\)-group. Then every irreducible representation over \(k\) of \(G\) is trivial (see \cite[\S 8.3]{serre}).
		It follows that every simple comodule of \((kG)^*\) is trivial, which shows that it is a connected Hopf algebra. Hence \((kG^*)_{par} = kG^*\) under these conditions. 
	\end{enumerate}
	\end{example}
	
	Let us end this section by remarking that a connected Hopf algebra is necessarily infinite-dimensional if the characteristic of \(k\) is 0.
	
	\begin{lemma}
		\label{le:connected_infinite}
		Let \(k\) be a field of characteristic 0 and \(H\) a finite-dimensional connected Hopf algebra over \(k\). Then \(H\) is trivial, i.\,e.\ \(\dim H = 1.\)
	\end{lemma}
	\begin{proof}
		Take \(x \in H_1\). Since \(\Delta(x) \in H \otimes H_0 + H_0 \otimes H\) by definition of \(H_1\), there exist \(y, z \in H\) such that
		\[\Delta(x) = y \otimes 1_H + 1_H \otimes z.\]
		Then \(x = \epsilon(y) 1_H + z = y + \epsilon(z) 1_H\) and
		\begin{align*}
			\Delta(z - \epsilon(z) 1_H) &= \Delta(x) - \epsilon(y) 1_H \otimes 1_H - \epsilon(z) 1_H \otimes 1_H \\
			&= (y - \epsilon(y) 1_H) \otimes 1_H + 1_H \otimes (z - \epsilon(z) 1_H) \\
			&= (z - \epsilon(z) 1_H) \otimes 1_H + 1_H \otimes (z - \epsilon(z) 1_H),
		\end{align*}
		so \(z - \epsilon(z) 1_H\) is a primitive element. By \cite[Corollary 9.1.2]{radford}, if \(k\) has characteristic zero, then the only primitive element in a finite-dimensional bialgebra over \(k\) is 0. Hence \(z - \epsilon(z) 1_H = 0\), which implies that \(x = \epsilon(y + z) 1_H \in H_0\). Hence \(H_1 = H_0,\) and it follows that \(H_n = H_0\) for all \(n \geq 0\). Now we conclude by \cref{le:coradical_filtration} that \(H = \bigcup_{n \geq 0} H_n = H_0 = k1_H\).
	\end{proof}

	\subsection{A partial converse}
	
	One might wonder if the converse of \cref{cor:connected} is true: if \(H\) is not connected, does it have a partial representation which is not global? The following proposition gives a positive answer if \(H\) is cosemisimple. 
	
	\begin{proposition}
		\label{le:cosemisimple_partial}
		Let \(H\) be a nontrivial cosemisimple Hopf algebra with invertible antipode. Then there exists a partial representation of \(H\) which is not global.
	\end{proposition}
	\begin{proof}
        By cosemisimplicity, \(H = k1_H \oplus C\) for some subcoalgebra \(C\) of \(H\). In fact, \(C\) is the (direct) sum of all simple subcoalgebras of \(H\) that do not contain \(1_H\).
        Moreover \(C\) is closed under the antipode, because if \(D\) is a simple subcoalgebra of \(H\) that does not contain \(1_H\), then so is \(S(D)\) by the invertibility of \(S\). Define a linear map \(\pi : H \to k\) by \(\pi(1_H) = 1\) and \(\pi(c) = 0\) for all \(c \in C\). It is easy to see that \(\pi\) is a partial representation, because in axiom \ref{PR2} both sides become zero if \(k \in C\), and reduce to \(\pi(h)\) if \(k = 1_H\). So \ref{PR2} holds for any \(h, k \in H = k1_H \oplus C\) and by reasoning on \(h\) instead of \(k,\) one sees that also \ref{PR3} holds. Take \(x \in C \setminus \ker \epsilon\). Such an element exists, because if not \(c = c_{(1)} \epsilon(c_{(2)}) = 0\) for all \(C,\) while we supposed that \(H\) is nontrivial. Now
		\[\pi(x_{(1)}) \pi(S(x_{(2)})) = 0 \neq  \epsilon(x) = \pi(x_{(1)} S(x_{(2)})),\]
		which shows that \(\pi\) is not a global representation.
	\end{proof}
	
	\begin{corollary}
		\label{cor:cosemisimple_partial}
		Let \(H\) be a Hopf algebra with invertible antipode which possesses a nontrivial cosemisimple Hopf quotient. Then there exists a partial representation of \(H\) which is not global.
	\end{corollary}
	
	We say a Hopf algebra has the \textit{Chevalley property} if the tensor product of any two simple modules is semisimple. If \(H\) is finite-dimensional, this is equivalent to saying that the Jacobson radical of \(H\) is a Hopf ideal, or that the coradical of \(H^*\) is a Hopf subalgebra by \cite[Proposition 4.2]{andruskiewitsch_chevalley}.
	
	\begin{proposition}
		Let \(k\) be an algebraically closed field of characteristic 0 and suppose that \(H\) is a nontrivial finite-dimensional Hopf algebra over \(k\) with the Chevalley property. Then \(H \neq H_{par}\).
	\end{proposition}
	\begin{proof}
		Since \(H\) has the Chevalley property, its Jacobson radical \(J(H)\) is a Hopf ideal. Hence \(H/J(H)\) is a semisimple Hopf algebra, which is cosemisimple by \cite{LRsemisimple}. If \(H/J(H)\) would be trivial, then \(H^*\) would be a connected Hopf algebra, because the coradical of \(H^*\) is \((H/J(H))^*\). But over a field of characteristic 0, there are no nontrivial finite-dimensional connected Hopf algebras by \cref{le:connected_infinite}, so \(H/J(H)\) must be nontrivial. The result now follows from \cref{cor:cosemisimple_partial}.
	\end{proof}
	
	\begin{remark}
	    Although the coradical of \(H\) can predict to some extent the existence of true partial representations, it does not say anything about the wildness of the partial representation theory or the nature of \(H_{par}\): if \(k\) is an infinite field of characteristic different from 2, then the group algebra \(kC_2\) is equal to its coradical and has only one simple partial representation which is not global (\cref{ex:C2}), while the Sweedler Hopf algebra also has coradical \(kC_2\) but has infinitely many non-isomorphic simple partial representations (\cref{ex:sweedler}). Indeed, consider the algebra maps
        \[H_{par} \cong k[x, y]/(xy) \times H \to k : (P(x) + \lambda y, z) \mapsto P(\alpha)\]
        for \(\alpha \in k\).
	\end{remark}

	\section{Smash products and partial representations}
	\label{se:smash}

 \subsection{Smash products of (Hopf) algebras}
    \label{se:intro_smash}
	
	Let \(U\) and \(H\) be \(k\)-algebras and \(R : H \otimes U \to U \otimes H\) a linear map. 
We will use the notation (summation understood)
\[R(h \otimes u) = u^R \otimes h^R \in U \otimes H.\]
If multiple copies of \(R\) are needed, we will also use \(u^r \otimes h^r, u^{\rho} \otimes h^{\rho}, u^{R_1} \otimes h^{R_2},\) etc.

By definition $U \#_R H $ is equal to $U\x H$ as $k$-vector space, and its elements will often be denoted as \(u \#_R h = u \otimes h\). It bears the multiplication given by
\begin{equation}\label{equation.multiplication.on.AB.elements}
	(u\#_R h)(u'\#_R h') = uu'^R\#_R h^Rh'.
\end{equation}

\begin{defi}
	\label{def:smash}
	If $U\#_R H$ is an associative $k$-algebra with unit $1_U\#_R 1_H$, then we call $U\#_R H$ a \textit{smash product}.
\end{defi}

We recall the necessary and sufficient conditions on \(R\) for \(U\#_R H\) to be a smash product, following \cite{CIMZsmash}.

\begin{defi}[{\cite[Definition 2.4, 2.6]{CIMZsmash}}] 
\label{def:multiplicative}
The linear map \(R : H \otimes U \to U \otimes H\) is said to be
	\begin{enumerate}[(i)] 
		\item \textit{left normal} if $R(h\x 1_U) = 1_U\x h$ for all \(h \in H\);
		\item \textit{right normal} if
		$R(1_H\x u) = u\x 1_H$ for all \(u \in U\);
		\item \textit{left multiplicative} if 
		\begin{equation}
			\label{eq:leftmult}
			u^R\x (hh')^R = (u^R)^r\x h^r(h')^R,
		\end{equation}
		for all \(u \in U\) and \(h, h' \in H\);
		\item \textit{right multiplicative} if
		\begin{equation}
			\label{eq:rightmult}
			(uu')^R\x h^R = u^R(u')^r\x (h^R)^r,
		\end{equation}
		for all \(u, u' \in U\) and \(h \in H\).
	\end{enumerate}
	We say \(R\) is \textit{normal} (resp.\ multiplicative) if it is both left and right normal (resp.\ multiplicative).
\end{defi}

\begin{theorem}[{\cite[Theorem 2.5]{CIMZsmash}}]
	\label{th:smash}
	Let \(U, H\) be algebras and \(R : H \otimes U \to U \otimes H\) a linear map. Then \(U \#_R H\) is a smash product if and only if \(R\) is normal and multiplicative. 
\end{theorem}

From the normality of \(R\) we get the following.
\begin{lemma}
	\label{le:algebra_inclusions}
	Let $U\#_R H$ be a smash product. Then the inclusions
	$$ \iota_U : U \to U\#_R H \ \ \textrm{ and } \ \ \iota_H:H \to U\#_R H $$
	are algebra maps. Moreover $u\#_R h = (u\#_R 1_H)(1_U\#_R h)$ for all \(u \in U, h \in H\).
\end{lemma}

Suppose now that \(U\) and \(H\) are bialgebras. Then the tensor product \(U \otimes H\) trivially carries a coalgebra structure: 
\[\Delta(u \otimes h) = (u_{(1)} \otimes h_{(1)}) \otimes (u_{(2)} \otimes h_{(2)}); \qquad \epsilon(u \otimes h) = \epsilon_U(u)\epsilon_H(h).\] If a smash product \(U \#_R H\) becomes a bialgebra for this coproduct and counit, then it is called an \textit{\(R\)-smash product}. 

\begin{theorem}[{\cite[Corollary 4.6]{CIMZsmash}}]
	\label{th:Rsmash}
	Let \(U\) and \(H\) be algebras and \(R : H \otimes U \to U \otimes H\) a linear map. Then \(U \#_R H\) is an \(R\)-smash product if and only if \(R\) is normal, multiplicative and a coalgebra map. 
	
	If moreover \(U\) and \(H\) are Hopf algebras with antipodes \(S_U\) and \(S_H\) respectively,
	then \(U \#_R H\) is a Hopf algebra with antipode
	\begin{equation}
		\label{eq:antipode_smash}
		S_{U \#_R H}(u \#_R h) = (R \circ (S_H \otimes S_U) \circ \tau_{U, H})(u \#_R h) = S_U(u)^R \#_R S_H(h)^R.
	\end{equation}
\end{theorem}
Using Sweedler notation, the fact that \(R\) is a coalgebra map reads
\begin{gather}
	\epsilon_U(u^R) \epsilon_H(h^R) = \epsilon_U(u) \epsilon_H(h), \label{eq:Reps} \\
	(u^R)_{(1)} \otimes (h^R)_{(1)} \otimes (u^R)_{(2)} \otimes (h^R)_{(2)} = (u_{(1)})^R \otimes (h_{(1)})^R \otimes (u_{(2)})^r \otimes (h_{(2)})^r, \label{eq:Rcoalg}
\end{gather}
for all \(u \in U, h \in H\). 
In \cite[Corollary 4.6]{CIMZsmash}, \(S_U\) and \(S_H\) are also required to satisfy the relation 
\begin{equation}
	\label{eq:antipode}
	S_U(u^R)^r \otimes S_H(h^R)^r = S_U(u) \otimes S_H(h)
\end{equation}
but this is in fact automatic: it can be shown easily using \cref{le:algebra_inclusions} that 
\[S_{U \#_R H}(u_{(1)} \#_R h_{(1)}) (u_{(2)} \#_R h_{(2)}) = \epsilon_U(u) \epsilon_H(h) 1_U \#_R 1_H = (u_{(1)} \#_R h_{(1)}) S_{U \#_R H}(u_{(2)} \#_R h_{(2)})\]
without resorting to \eqref{eq:antipode} (see e.\,g.\ \cite[Proposition 4.7]{bulacu}). This shows that \(S_{U\#_R H}\) is an antipode (hence anti-multiplicative), and \eqref{eq:antipode} can be deduced by expressing \(S_{U\#_H H}(u^R \#_R h^R)\), using normality of \(R\).

As explained in \cite[Remark 4.4]{CIMZsmash}, Majid already constructed special cases of smash products in the sense of \cref{def:smash} in \cite{majid}, and the \(R\)-smash products we consider correspond exactly to Majid's double crossproducts \(U \bowtie H\), which arise from a matched pair of Hopf algebras \cite[Definition 7.2.1]{majid}. The following proposition explains how to obtain a matched pair from an \(R\)-smash product and vice versa. A more general version (which also twists the comultiplication) appeared in \cite[Theorem 5.4]{bulacu}.

\begin{proposition}
	\label{prop:actions}
	Let \(R\) be as in \cref{th:Rsmash}. Then 
	\begin{equation*}
		H \otimes U \to U : h \otimes u \mapsto h \triangleright u = \epsilon_H(h^R) u^R
	\end{equation*}
	defines a left \(H\)-module structure on \(U\), and
	\begin{equation*}
		H \otimes U \to H : h \otimes u \mapsto h \triangleleft u = \epsilon_U(u^R) h^R
	\end{equation*}
	defines a right \(U\)-module structure on \(H\). 
	With these actions, \((U, H)\) forms a matched pair of Hopf algebras. 
	Moreover, for any \(h \otimes u \in H \otimes U,\)
	\begin{equation}
		\label{eq:Rbyaction}
		u^R \otimes h^R = h_{(1)} \triangleright u_{(1)} \otimes h_{(2)} \triangleleft u_{(2)} =  h_{(2)} \triangleright u_{(2)} \otimes h_{(1)} \triangleleft u_{(1)}.
	\end{equation}
\end{proposition}
\begin{proof}
	The fact that \(\triangleright\) and \(\triangleleft\) are actions follows from the normality and multiplicativity of \(R\). The equalities in \eqref{eq:Rbyaction} are obtained by applying \(\epsilon_U\) and \(\epsilon_H\) on the first and the fourth (resp.\ the third and the second) tensorand of \eqref{eq:Rcoalg}. 
\end{proof}

\begin{example}
	\label{ex:hopfactionR}
	If the right \(U\)-action \(\triangleleft\) is trivial, i.\,e.\ \(h \triangleleft u = \epsilon_U(u) h\) for all \(u \in U, h \in H\) then we recover the most well-known form of smash product:
	\[R(h \otimes u) = h_{(1)} \triangleright u \otimes h_{(2)},\]
	where \(U\) is now an \(H\)-module algebra by \(\triangleright,\) compatible with the counit and comultiplication. If \(S_H\) is invertible, then so is \(R:\) its inverse is
	\[R'(u \otimes h) = h_{(2)} \otimes S^{-1}(h_{(1)}) \triangleright u.\]
\end{example}

We end this section with some technical lemmas for later use, which appeared already in the PhD thesis of the first author \cite{tiagothesis}. We also determine the inverse of \(R\) in the general case, assuming that \(S_U\) and \(S_H\) are invertible.

\begin{lemma}
	For any \(u \in U, h \in H,\)
	\begin{align}
		S_U(h \triangleright u) &= (h \triangleleft u_{(1)}) \triangleright S_U(u_{(2)}), \label{eq:Striangle1} \\ 
		S_H(h \triangleleft u) &= S_H(h_{(1)}) \triangleleft (h_{(2)} \triangleright u). \label{eq:Striangle2}
	\end{align}
\end{lemma}
\begin{proof}
	It is easy to see that \(S_U \circ\, \triangleright\) is the inverse of \(\triangleright\) in the convolution algebra \(\mathsf{Hom}_k(H \otimes U, U)\). One can calculate, using the right multiplicativity and the comultiplicativity of \(R\), that for any \(u \in U, h \in H,\)
	\[(h_{(1)} \triangleright u_{(1)})((h_{(2)} \triangleleft u_{(2)}) \triangleright S_U(u_{(3)})) = \epsilon_U(u) \epsilon_H(h) 1_U,\]
	so \(\triangleright \circ (\triangleleft \otimes S_U) \circ (I_H \otimes \Delta_U)\) is a right convolution inverse of \(\triangleright\). Being invertible, the right inverse of \(\triangleright\) is unique, and \eqref{eq:Striangle1} follows. Identity \eqref{eq:Striangle2} is shown in a similar way. 
\end{proof}

\begin{lemma}
	\label{le:W}
	For any \(u \in U, h \in H,\)
	\begin{gather}
		(u^r)^R\x S_H(h_{(1)})^R\x (h_{(2)})^r = u_{(1)} \x S_H((h \triangleleft u_{(2)})_{(1)})\x (h \triangleleft u_{(2)})_{(2)},\label{eq:W1} \\
		{u_{(1)}}^R \otimes S_U(u_{(2)})^r \otimes (h^{R})^{r} = (h_{(1)} \triangleright u)_{(1)} \otimes S_U((h_{(1)} \triangleright u)_{(2)}) \otimes h_{(2)}. \label{eq:W2}
	\end{gather}
\end{lemma}
\begin{proof}
	One should rewrite the left-hand side using the actions \(\triangleleft\) and \(\triangleright,\) and apply \eqref{eq:Striangle2} (resp.\ \eqref{eq:Striangle1}) to obtain  formula \eqref{eq:W1} (resp.\ \eqref{eq:W2}).
\end{proof}

\begin{lemma}
    Suppose that \(S_U\) and \(S_H\) are invertible. If \(R\) is normal, multiplicative and a coalgebra map, then it is invertible, with inverse
\begin{equation}
	\label{eq:R'}
R' : U \otimes H \to H \otimes U : u \otimes h \mapsto S_H(S_H^{-1}(h_{(1)}) \triangleleft S_U^{-1}(u_{(1)})) \otimes S_U(S_H^{-1}(h_{(2)}) \triangleright S_U^{-1}(u_{(2)})).
\end{equation}
\end{lemma}
\begin{proof}
    It follows from \eqref{eq:Striangle2} that 
    \[S_H(h_{(1)} \triangleleft (S_H(h_{(2)}) \triangleright u)) = S_H(h_{(1)}) \triangleleft (h_{(2)} \triangleright (S_H(h_{(3)}) \triangleright u) 
        = S_H(h) \triangleleft u,\]
    and by applying \(S_H^{-1}\) to both sides and replacing \(h\) by \(S_H^{-1}(h),\) we get
    \begin{equation}
        \label{eq:Sinvtriangle}
        S_H^{-1}(h \triangleleft u) = S_H^{-1}(h_{(2)}) \triangleleft (h_{(1)} \triangleright u).
    \end{equation}
    Similarly,
    \[S_U^{-1}(h \triangleright u) = (h \triangleleft u_{(2)}) \triangleright S_U^{-1}(u_{(1)}).\]
    Now we calculate
    \begin{align*}
        R'R(h \otimes u) &\stackrel{\phantom{\eqref{eq:Rbyaction}}}{=} S_H(S_H^{-1}(h_{(3)} \triangleleft u_{(3)}) \triangleleft S_U^{-1}(h_{(1)} \triangleright u_{(1)})) \otimes S_U(S_H^{-1}(h_{(4)} \triangleleft u_{(4)}) \triangleright S_U^{-1}(h_{(2)} \triangleright u_{(2)})) \\ 
        &\stackrel{\eqref{eq:Rbyaction}}{=} S_H(S_H^{-1}(h_{(2)} \triangleleft u_{(2)}) \triangleleft S_U^{-1}(h_{(1)} \triangleright u_{(1)})) \otimes S_U(S_H^{-1}(h_{(4)} \triangleleft u_{(4)}) \triangleright S_U^{-1}(h_{(3)} \triangleright u_{(3)})) \\
        &\stackrel{\eqref{eq:Sinvtriangle}}{=} S_H((S_H^{-1}(h_{(3)}) \triangleleft (h_{(2)} \triangleright u_{(2)})) \triangleleft S_U^{-1}(h_{(1)} \triangleright u_{(1)})) \\ & \qquad \qquad \qquad \qquad \otimes S_U(S_H^{-1}(h_{(5)} \triangleleft u_{(5)}) \triangleright ((h_{(4)} \triangleleft u_{(4)}) \triangleright S_U^{-1}(u_{(3)}))) \\
        &\stackrel{\phantom{\eqref{eq:Rbyaction}}}{=} h_{(2)} \epsilon_H(h_{(1)}) \epsilon_U(u_{(1)}) \otimes \epsilon_U(u_{(3)}) \epsilon_H(h_{(3)}) u_{(2)} \\
        &\stackrel{\phantom{\eqref{eq:Rbyaction}}}{=} h \otimes u,
    \end{align*}
    and in the same way \(RR'(u \otimes h) = u \otimes h\).
\end{proof}

Clearly \(R'\) is normal because \(R\) is. From the left multiplicativity of \(R,\) we get
\begin{align*}
	R(h^{r'} h'^{R'} \otimes u^{r'R'}) &= u^{r'R'R} \otimes (h^{r'} h'^{R'})^R \\
	&= u^{r'R'Rr} \otimes h^{r'r} h'^{R'R} \\
	&= u \otimes hh',
\end{align*} 
hence \(h^{r'} h'^{R'} \otimes a^{r'R'} = (hh')^{R'} \otimes a^{R'}\) which means that \(R'\) is right multiplicative. Similarly it follows that \(R'\) is left multiplicative from the right multiplicativity of \(R\). Finally, \(R'\) is a coalgebra map because \(R\) is. Since \(R'\) satisfies the same properties as \(R\) with \(H\) and \(U\) swapped, we can write down the statements corresponding to \cref{prop:actions} and \cref{le:W} for \(R'\):
\begin{itemize}
	\item \(H\) is a left \(U\)-module, and \(U\) is a right \(H\)-module via the actions
	\begin{align*}
	    u \blacktriangleright h &= \epsilon_U(u^{R'}) h^{R'} = S_H(S^{-1}_H(h) \triangleleft S_U^{-1}(u)), \\
        u \blacktriangleleft h &= \epsilon_H(h^{R'}) u^{R'} = S_U(S^{-1}_H(h) \triangleright S_U^{-1}(u)).
	\end{align*}
	\item These actions recover \(R'\) since
	\[h^{R'} \otimes u^{R'} = u_{(1)} \blacktriangleright h_{(1)} \otimes u_{(2)} \blacktriangleleft h_{(2)}.\]
	\item We have
	\begin{align}
		\label{eq:W'}
		h^{r'R'} \otimes S_U(u_{(1)})^{R'} \otimes u_{(2)}^{r'} &= h_{(1)} \otimes S_U((u \blacktriangleleft h_{(2)})_{(1)}) \otimes (u \blacktriangleleft h_{(2)})_{(2)}, \\
		{h_{(1)}}^{R'} \otimes S_H(h_{(2)})^{r'} \otimes u^{R'r'} &= (u_{(1)} \blacktriangleright h)_{(1)} \otimes S_H((u_{(1)} \blacktriangleright h)_{(2)}) \otimes u_{(3)}.
	\end{align}
\end{itemize}
	
	\subsection{Main results}
	\label{se:main}
	
Let \(U\) and \(H\) be Hopf algebras and \(R : H \otimes U \to U \otimes H\) be an invertible linear map so that \(U \#_R H\) is an \(R\)-smash product (cf. \cref{th:Rsmash}). Consider the following categories:
\begin{itemize}
	\item the category \(\mathcal{P}_{U, H}^R\) with objects \textit{pairs of partial representations} \((\alpha : U \to B, \eta : H \to B)\) (where \(B\) is any \(k\)-algebra) such that 
\begin{equation}\label{eq:partial.pairs}
    \eta(h) \alpha(u) = \alpha(u^R) \eta(h^R)
\end{equation}	
	for all \(u \in U, h \in H,\)
	and morphisms \textit{algebra maps} \(f : B \to B'\) such that \(f \circ \alpha = \alpha'\) and \(f \circ \eta = \eta'\);
	\item the category \(\PRep_{U \#_R H}^{\mathrm{split}}\) with objects \textit{partial representations} \(\pi : U \#_R H \to B\) such that
	\begin{align}
		\pi(u \#_R h) = \pi(u \#_R 1_H) \pi(1_U \#_R h), \label{eq:split1} \\
		\pi(u^R \# h^R) =  \pi(1_U \#_R h) \pi(u \#_R 1_H), \label{eq:split2}
	\end{align}
	for all \(u \in U, h \in H\), and morphisms \textit{algebra maps} \(f : B \to B'\) such that \(f \circ \pi = \pi'\).
\end{itemize}
We will show that these categories are isomorphic, and that they are furthermore isomorphic to the category of representations of a smash product of the partial ``Hopf'' algebras \(U_{par}\) and \(H_{par}\), which we will now describe.

Define a linear map \(\hat{\mathcal{R}} : T(H) \otimes T(U) \to U_{par} \otimes H_{par}\) by
\begin{equation*}
     \hat{\mathcal{R}}((h_1 \otimes \cdots \otimes h_n) \otimes (u_1 \otimes \cdots \otimes u_m)) = [{u_1}^{R_{11} \cdots R_{1n}}] \cdots [{u_m}^{R_{m1} \cdots R_{mn}}] \otimes [{h_1}^{R_{1n} \cdots R_{mn}}] \cdots [{h_n}^{R_{11} \cdots R_{m1}}].
\end{equation*}
\begin{lemma}
\label{le:Rhat}
    For any \(h, k, h_1, \dots h_n \in H\) and \(u, v, u_1, \dots, u_m \in U\),
    \begin{align*}
		\hat{\mathcal{R}}((S(h_{(1)}) \otimes  h_{(2)} \otimes k) \x (u_1 \otimes \cdots \otimes u_m)) &= \hat{\mathcal{R}}((S(h_{(1)}) \otimes  h_{(2)} k) \x (u_1 \otimes \cdots \otimes u_m)), \\
		\hat{\mathcal{R}}((k \otimes S(h_{(1)}) \otimes h_{(2)}) \x (u_1 \otimes \cdots \otimes  u_m)) &= 
		\hat{\mathcal{R}}((k S(h_{(1)}) \otimes h_{(2)}) \x (u_1 \otimes \cdots \otimes  u_m)), \\
		\hat{\mathcal{R}}((h_1 \otimes \cdots \otimes h_n) \otimes (u_{(1)} \otimes S(u_{(2)}) \otimes v)) &= 
		\hat{\mathcal{R}}((h_1 \otimes \cdots \otimes h_n) \otimes (u_{(1)} \otimes S(u_{(2)}) v)), \\
		\hat{\mathcal{R}}((h_1 \otimes \cdots \otimes h_n) \otimes (v \otimes u_{(1)} \otimes S(u_{(2)}))) &=
		\hat{\mathcal{R}}((h_1 \otimes \cdots \otimes h_n) \otimes (v u_{(1)} \otimes S(u_{(2)}))).
	\end{align*}
\end{lemma}
\begin{proof}
    Let us check the first equality, the others are shown in a similar way. We have 
	\begin{align*}
		&\!\!\!\!\!\!\!\hat{\mathcal{R}}((S(h_{(1)}) \otimes  h_{(2)} \otimes k) \x (u_1 \otimes \cdots \otimes u_m)) \\ & \stackrel{\mathmakebox[\widthof{   =   }]{}}{=} [u_1^{R_{11} R_{12} R_{13}}] \cdots [u_m^{R_{m1} R_{m2} R_{m3}}] \otimes [S(h_{(1)})^{R_{13} \cdots R_{m3}}][{h_{(2)}}^{R_{12} \cdots R_{m2}}] [k^{R_{11} \cdots R_{m1}}] \\
		&\stackrel{\mathmakebox[\widthof{   =   }]{\eqref{eq:W1}}}{=} [{u_{1}^{R_{11}}}_{(1)}] \cdots [{u_{m}^{R_{m1}}}_{(1)}] \\ & \qquad \qquad \otimes [S((h \triangleleft {u_{1}^{R_{11}}}_{(2)} \cdots {u_{m}^{R_{m1}}}_{(2)})_{(1)})] [(h \triangleleft {u_{1}^{R_{11}}}_{(2)} \cdots {u_{m}^{R_{m1}}}_{(2)})_{(2)})] [k^{R_{11} \cdots R_{m1}}] \\
		&\stackrel{\mathmakebox[\widthof{   =   }]{\ref{PR5}}}{=} [{u_{1}^{R_{11}}}_{(1)}] \cdots [{u_{m}^{R_{m1}}}_{(1)}] \\ & \qquad \qquad \otimes [S((h \triangleleft {u_{1}^{R_{11}}}_{(2)} \cdots {u_{m}^{R_{m1}}}_{(2)})_{(1)})] [(h \triangleleft {u_{1}^{R_{11}}}_{(2)} \cdots {u_{m}^{R_{m1}}}_{(2)})_{(2)})k^{R_{11} \cdots R_{m1}}] \\
		&\stackrel{\mathmakebox[\widthof{   =   }]{}}{=} [{u_1}^{R_{11} R_{12} R_{13}}] \cdots [{u_m}^{R_{m1} R_{m2} R_{m3}}] \otimes [S(h_{(1)})^{R_{13} \cdots R_{m3}}][{h_{(2)}}^{R_{12} \cdots R_{m2}}k^{R_{11} \cdots R_{m1}}] \\
		&\stackrel{\mathmakebox[\widthof{   =   }]{\eqref{eq:leftmult}}}{=} [{u_1}^{R_{12} R_{13}}] \cdots [{u_m}^{R_{m2} R_{m3}}] \otimes [S(h_{(1)})^{R_{13} \cdots R_{m3}}][(h_{(2)}k)^{R_{12} \cdots R_{m2}}] \\
		&\stackrel{\mathmakebox[\widthof{   =   }]{}}{=} \hat{\mathcal{R}}((S(h_{(1)}) \otimes  h_{(2)} k) \x (u_1 \otimes \cdots \otimes u_m)). \qedhere
	\end{align*}
\end{proof}
\begin{lemma}
	\label{le:mathcalR}
	The map \(\hat{\mathcal{R}}\) induces a normal and multiplicative map
 \begin{align*}
     \mathcal{R} : H_{par} \otimes U_{par} &\to U_{par} \otimes H_{par}, \\
     [h_1] \cdots [h_n] \otimes [u_1] \cdots [u_m] &\mapsto [{u_1}^{R_{11} \cdots R_{1n}}] \cdots [{u_m}^{R_{m1} \cdots R_{mn}}] \otimes [{h_1}^{R_{1n} \cdots R_{mn}}] \cdots [{h_n}^{R_{11} \cdots R_{m1}}].
 \end{align*}
\end{lemma}
\begin{proof}
    Let us first check that \(\hat{\mathcal{R}}\) is multiplicative in the sense of \cref{def:multiplicative}. 
\begin{align*}
		& \!\!\!\!\!\!\! \hat{\mathcal{R}}((h_1 \otimes \cdots \otimes h_n \otimes h'_1 \otimes \cdots \otimes h'_{n'}) \otimes (u_1 \otimes \cdots \otimes u_m)) \\
		&= [u_1^{R_{11} \cdots R_{1, n + n'}}] \cdots [u_m^{R_{m1} \cdots R_{m, n+ n'}}] \\ 
  & \qquad \qquad \otimes [h_1^{R_{1, n + n'} \cdots R_{m, n + n'}}] \cdots [h_n^{R_{1,n' + 1} \cdots R_{m,n' + 1}}] [{h'_1}^{R_{1n'} \cdots R_{mn'}}] \cdots [{h'_{n'}}^{R_{11} \cdots R_{m1}}] \\
		&= (u_1 \otimes \cdots \otimes u_m)^{\mathcal{R}' \mathcal{R}} \otimes (h_1 \otimes \cdots \otimes h_n)^\mathcal{R} (h'_1 \otimes \cdots  \otimes h'_{n'})^{\mathcal{R}'},
	\end{align*}
 and a similar computation shows that it is right multiplicative. Together with \cref{le:Rhat}, this shows that the kernel of \(\hat{\mathcal{R}}\) contains \(\mathcal{I}_H \otimes T(U) + T(H) \otimes \mathcal{I}_U\), where \(\mathcal{I}_H\) and \(\mathcal{I}_U\) are the ideals defined by the relations \eqref{eq:Hparrelation1}, \eqref{eq:Hparrelation2} and \eqref{eq:Hparrelation3}, i.\,e.\ \(H_{par} = T(H)/\mathcal{I}_H\) and \(U_{par} = T(U)/\mathcal{I}_U\). 
Hence \(\hat{\mathcal{R}}\) induces a linear map \(\mathcal{R} : H_{par} \otimes U_{par} \to U_{par} \otimes H_{par}\). The computation above shows that it is multiplicative, and the normality is obvious too.    
\end{proof}

By \cref{th:smash}, we obtain an algebra \(U_{par} \#_{\mathcal{R}} H_{par}\) and we are ready to state our main result. 

\begin{theorem}
	\label{th:isosP}
	Let \(U, H\) be Hopf algebras with invertible antipode and \(R : H \otimes U \to U \otimes H\) be normal, multiplicative and a coalgebra map. Then there are isomorphisms of categories
	\[\mathcal{P}^R_{U, H} \simeq \PRep^{\mathrm{split}}_{U \#_R H} \simeq \Rep_{U_{par} \#_{\mathcal{R}} H_{par}}.\]
\end{theorem}

\begin{proof}
	We start with the first isomorphism.
	Let \(\pi : U \#_R H \to B\) be a partial representation such that \(\pi(u \#_R h) = \pi(u \#_R 1_H) \pi(1_U \#_R h)\) and \(\pi(u^R \#_R h^R) = \pi(1_U \#_R h)\pi(u \#_R 1_H)\) for all \(u \#_R h \in U \#_R H\). Then 
	\begin{align*}
		\alpha_\pi : U \to B : u \mapsto \pi(u \#_R 1_H) \\
		\eta_\pi : H \to B : h \mapsto \pi(1_U \#_R h)
	\end{align*}
	are partial representations, because the canonical inclusions \(\iota_U : U \to U \#_R H\) and \(\iota_H : H \to U \#_R H\) are morphisms of Hopf algebras. Moreover, 
	\begin{align*}
	    \alpha_\pi(u^R) \eta_\pi(h^R) &= \pi(u^R \#_R 1_H) \pi(1_U \#_R h^R) = \pi(u^R \#_R h^R) \\ &= \pi(1_U \#_R h) \pi(u \#_R 1_H) = \eta_\pi(h) \alpha_\pi(u).
	\end{align*}
	
	Conversely, let \(\alpha : U \to B\) and \(\eta : H \to B\) be partial representations such that \(\eta(h) \alpha(u) = \alpha(u^R) \eta(h^R)\). We will show that
	\[\pi_{\alpha, \eta} : U \#_R H \to B : u \#_R h \mapsto \alpha(u) \eta(h)\]
	is a partial representation. Recall that \(R\) is invertible with inverse \(R'\) \eqref{eq:R'}. So we have also that
	\begin{equation}
		\label{eq:R'alphaeta}
		\alpha(u) \eta(h) = \eta(h^{R'}) \alpha(u^{R'})
	\end{equation} 
	for all \(u \in U, h \in H\).
	Then,
    \begingroup
    \allowdisplaybreaks
	\begin{align*}
		\pi_{\alpha, \eta}&(v\#_R k)\pi_{\alpha, \eta}(S(u_{(1)})^R\#_R S(h_{(1)})^R)\pi_{\alpha, \eta}(u_{(2)}\#_R h_{(2)}) \\
		&\stackrel{\mathmakebox[\widthof{   =   }]{}}{=} \alpha(v)\eta(k)\alpha(S(u_{(1)})^R)\eta(S(h_{(1)})^R) \alpha(u_{(2)})\eta(h_{(2)}) \\
		&\stackrel{\mathmakebox[\widthof{   =   }]{ \eqref{eq:partial.pairs}}}{=} \alpha(v)\eta(k)\eta(S(h_{(1)}))\alpha(S(u_{(1)})) \alpha(u_{(2)})\eta(h_{(2)}) \\
		&\stackrel{\mathmakebox[\widthof{   =   }]{ \eqref{eq:R'alphaeta}}}{=} \alpha(v)\eta(k)\eta(S(h_{(1)})) \eta({h_{(2)}}^{R'r'})\alpha(S(u_{(1)})^{r'}) \alpha({a_{(2)}}^{R'}) \\
		&\stackrel{\mathmakebox[\widthof{   =   }]{\eqref{eq:W'}}}{=} \alpha(v)\eta(k)\eta(S(h_{(1)})) \eta({h_{(2)}})\alpha(S((u \blacktriangleleft h_{(3)})_{(1)})) \alpha({(u \blacktriangleleft h_{(3)})_{(2)}}) \\
		&\stackrel{\mathmakebox[\widthof{   =   }]{\ref{PR4}}}{=} \alpha(v)\eta(kS(h_{(1)})) \eta({h_{(2)}})\alpha(S((u \blacktriangleleft h_{(3)})_{(1)})) \alpha({(u \blacktriangleleft h_{(3)})_{(2)}}) \\
		&\stackrel{\mathmakebox[\widthof{   =   }]{ \eqref{eq:R'alphaeta}}}{=} \eta((kS(h_{(1)}))^{R'}) \eta({h_{(2)}}^{r'})\alpha(v^{R'r'})\alpha(S((u \blacktriangleleft h_{(3)})_{(1)})) \alpha({(u \blacktriangleleft h_{(3)})_{(2)}}) \\
		&\stackrel{\mathmakebox[\widthof{   =   }]{\ref{PR4}}}{=} \eta((kS(h_{(1)}))^{R'}) \eta({h_{(2)}}^{r'})\alpha(v^{R'r'}S((u \blacktriangleleft h_{(3)})_{(1)})) \alpha({(u \blacktriangleleft h_{(3)})_{(2)}}) \\
		&\stackrel{\mathmakebox[\widthof{   =   }]{ \eqref{eq:partial.pairs}}}{=} \alpha([b^{R'r'}S((u \blacktriangleleft h_{(3)})_{(1)})]^{rR}) \eta((kS(h_{(1)}))^{R'R}) \eta({h_{(2)}}^{r'r}) \alpha({(u \blacktriangleleft h_{(3)})_{(2)}}) \\
		&\stackrel{\mathmakebox[\widthof{   =   }]{\eqref{eq:rightmult}}}{=} \alpha(v^{R'r'r_1R_1}S((u \blacktriangleleft h_{(3)})_{(1)})^{r_2R_2}) \eta((kS(h_{(1)}))^{R'R_1R_2}) \eta({h_{(2)}}^{r'r_1r_2}) \alpha({(u \blacktriangleleft h_{(3)})_{(2)}}) \\
		&\stackrel{\mathmakebox[\widthof{   =   }]{}}{=} \alpha(vS((u \blacktriangleleft h_{(3)})_{(1)})^{rR}) \eta((kS(h_{(1)}))^{R}) \eta({h_{(2)}}^{r}) \alpha({(u \blacktriangleleft h_{(3)})_{(2)}}) \\
		&\stackrel{\mathmakebox[\widthof{   =   }]{\eqref{eq:W'}}}{=} \alpha(vS(u_{(1)})^{r'rR}) \eta((kS(h_{(1)}))^{R}) \eta({h_{(2)}}^{R'r'r}) \alpha({u_{(2)}}^{R'}) \\
        &\stackrel{\mathmakebox[\widthof{   =   }]{}}{=} \alpha(vS(u_{(1)})^{R}) \eta((kS(h_{(1)}))^{R}) \eta({h_{(2)}}^{R'}) \alpha({u_{(2)}}^{R'}) \\
		&\stackrel{\mathmakebox[\widthof{   =   }]{ \eqref{eq:R'alphaeta}}}{=} \alpha(vS(u_{(1)})^R)\eta((kS(h_{(1)}))^R)\alpha(u_{(2)})\eta(h_{(2)}) \\
		&\stackrel{\mathmakebox[\widthof{   =   }]{\eqref{eq:leftmult}}}{=} \alpha(vS(u_{(1)})^{Rr})\eta(k^rS(h_{(1)})^R) \alpha(u_{(2)})\eta(h_{(2)}) \\
		&\stackrel{\mathmakebox[\widthof{   =   }]{}}{=} \pi_{\alpha, \eta}(v(S(u_{(1)})^{Rr}\#_R k^rS(h_{(1)})^R)) \pi_{\alpha, \eta}(u_{(2)}\#_R h_{(2)}) \\
		&\stackrel{\mathmakebox[\widthof{   =   }]{}}{=} \pi_{\alpha, \eta}((v\#_R k)(S(u_{(1)})^R\#_R S(h_{(1)})^R))\pi(u_{(2)}\#_R h_{(2)}),
	\end{align*}
    \endgroup
	proving \ref{PR4}. Using similar steps, one shows that \(\pi_{\alpha, \eta}\) also satisfies \ref{PR5}. 
		
	This gives indeed a bijective correspondance because \(\pi(u \#_R h) = \pi(u \#_R 1_H) \pi(1_U \#_R h)\) by assumption. Functoriality is clear. 
	
	Now the isomorphism between \(\mathcal{P}^R_{U, H}\) and \(\Rep_{U_{par} \#_{\mathcal{R}} H_{par}}\). Let \(\sigma : U_{par} \#_\mathcal{R} H_{par} \to B\) be an algebra map. We have partial representations
	\begin{align*}
		U \to U_{par} \#_\mathcal{R} H_{par} &: u \mapsto [u] \#_{\mathcal{R}} [1_H],  \\
		H \to  U_{par} \#_\mathcal{R} H_{par} &: h \mapsto [1_U] \#_{\mathcal{R}} [h]
	\end{align*}
	(this follows from the fact that \(\mathcal{R}\) is normal and that \([-] : U \to U_{par},\) resp.\ \([-] : H \to H_{par}\) are partial representations). They are compatible through \(R,\) so it is easy to see that 
	\begin{align*}
		\alpha &: U \to B : u \mapsto \sigma([u] \#_{\mathcal{R}} [1_H]), \\
		\eta &: H \to B : h \mapsto \sigma([1_U] \#_{\mathcal{R}} [h])
	\end{align*}
	constitutes an object in \(\mathcal{P}^R_{U, H}\).
	
	Conversely, let \(\alpha : U \to B, \eta : H \to B\) be  partial representations such that \(\eta(h) \alpha(u) = \alpha(u^R) \eta(h^R)\) for all \(u \in U, h \in H\). Define
	\[\sigma_{\alpha, \eta} : U_{par} \#_{\mathcal{R}} H_{par} \to B: [u_1] \cdots [u_m] \otimes [h_1] \cdots [h_n] \mapsto \alpha(u_1) \cdots \alpha(u_m) \eta(h_1) \cdots \eta(h_n).\]
	It is an easy verification that this is an algebra morphism, and that the obtained correspondence is bijective. 
\end{proof}

\begin{remark}
	In general, \(\PRep_{U \#_R H}^{\mathrm{split}}\) is a proper subcategory of \(\PRep_{U \#_R H}\), and there might be many partial representations of \(U \#_R H\) that do not satisfy \eqref{eq:split1} and \eqref{eq:split2}. As an elementary example, let \(C_2 = \{1, g\}\) be the cyclic group of order two and \(G = C_2 \times C_2\). Let \(k\) be a field of characteristic different from 2. Then \(kG \cong kC_2 \otimes kC_2\), which is an \(R\)-smash product with trivial \(R\) (i.\,e.\ both actions \(\triangleleft,\, \triangleright\) are trivial, which means that \(R\) is just the twist map). However \(\dim k_{par} G = 20\) (see \cite{DEP}) while \(\dim kC_{2,par} = 3,\) so that \(\dim kC_{2,par} \otimes kC_{2,par} = 9\). Concretely, the 3-dimensional irreducible partial representation of \(G\) is not in \(\PRep_{kC_2 \otimes kC_2}^{\mathrm{split}}\) as well as two 1-dimensional partial representations obtained from global representations of the diagonal subgroup \(\{(1, 1), (g, g)\}\) of \(G\).
\end{remark}

The category \(\mathcal{P}_{U, H}^R\) has an interesting subcategory: the category \(\mathcal{Q}_{U, H}^R\) of pairs \((\alpha, \eta)\) where \(\alpha : U \to B\) is an algebra map and \(\eta : H \to B\) is a partial representation such that \(\eta(h) \alpha(u) = \alpha(u^R) \eta(h^R)\). Under the isomorphisms of categories from \cref{th:isosP},  \(\mathcal{Q}_{U, H}^R\) corresponds to the category \(\PRep^{\mathrm{glob}}_{U \#_R H}\) of partial representations \(\pi : U \#_R H \to B\) such that \
\begin{equation}
	\label{eq:piAmult}
	\pi(uu' \#_R 1_H) = \pi(u \#_R 1_H) \pi(u' \#_R 1_H).
\end{equation} If \(\pi\) satisfies \eqref{eq:piAmult}, then 
(using that \(S(u \#_R 1_H) = R(1_H \#_R S(u)) = S(u) \#_R 1_H\)),
\begin{align*}
	\pi(u \#_R 1_H) \pi(1_U \#_R h) &= \pi(u_{(1)} \#_R 1_H) \pi(S(u_{(2)}) \#_R 1_H) \pi(u_{(3)} \#_R 1_H) \pi(1_U \#_R h) \\
	&= \pi(u_{(1)} \#_R 1_H) \pi(S(u_{(2)}) \#_R 1_H) \pi(u_{(3)} \#_R h) \\
	&= \pi(u \#_R h)
\end{align*}
and
\begin{align*}
	\pi(1_U \#_R h)\pi(u \#_R 1_H) &= \pi(1_U \#_R h) \pi(u_{(1)} \#_R 1_H) \pi(S(u_{(2)}) \#_R 1_H) \pi(u_{(3)} \#_R 1_H) \\
	&= \pi({a_{(1)}}^R \#_R h^R) \pi(S(u_{(2)}) \#_R 1_H) \pi(u_{(3)} \#_R 1_H) \\
	&= \pi(u^R \#_R h^R),
\end{align*}
so \(\PRep^{\mathrm{glob}}_{U \#_R H}\) is indeed a subcategory of \(\PRep^{\mathrm{split}}_{U \#_R H}\). These subcategories can also be described as the category of representations of a smash product of \(U\) and \(H_{par},\) by the local braiding
\[\mathcal{T} : H_{par} \otimes U \to U \otimes H_{par} : [h_1] \cdots [h_n] \otimes u \mapsto u^{R_1 \cdots R_n} \otimes [{h_1}^{R_n}] \cdots [{h_n}^{R_1}].\]
It can be seen from \cref{le:Rhat} and \cref{le:mathcalR} that \(\mathcal{T}\) is well-defined, normal and multiplicative. 

\begin{theorem}
	\label{th:smashglobalpartial}
	Let \(U, H\) be Hopf algebras and let \(R : H \otimes U \to U \otimes H\) be normal, multiplicative and a coalgebra map. Then there are isomorphisms of categories
	\[\mathcal{Q}_{U, H}^R \simeq \PRep_{U \#_R H}^{\mathrm{glob}} \simeq \Rep_{U \#_{\mathcal{T}} H_{par}}.\]
\end{theorem}
\begin{proof}
	Let us just remark that it is not necessary to suppose that \(R\) is invertible to prove the first isomorphism. The rest of the proof is identical to the proof of \cref{th:isosP}.
	
	Let \(\alpha : U \to B\) be an algebra map and let \(\eta : H \to B\) be a partial representation such that \(\eta(h) \alpha(u) = \alpha(u^R) \eta(h^R)\). Define $\pi_{\alpha, \eta} :U\#_R H \to X$, given by $\pi_{\alpha, \eta}(u\# h) = \alpha(u)\eta(h)$.
	
	Let us show that \(\pi_{\alpha, \eta}\) satisfies \ref{PR4}, the proof for \ref{PR5} is analogous. 
	\begin{align*}
		\pi_{\alpha, \eta}(v\# k)\pi_{\alpha, \eta}(S(u_{(1)})^R\# S(h_{(1)})^R)\pi_{\alpha, \eta}(u_{(2)}\# h_{(2)}) \span \\
		\qquad \qquad \qquad &\stackrel{\mathmakebox[\widthof{   =   }]{}}{=}  \alpha(v)\eta(k)\alpha(S(u_{(1)})^R)\eta(S(h_{(1)})^R)\alpha(u_{(2)})\eta(h_{(2)}) \\
		&\stackrel{\mathmakebox[\widthof{   =   }]{}}{=} \alpha(v)\eta(k)\eta(S(h_{(1)}))\alpha(S(u_{(1)}))\alpha(u_{(2)})\eta(h_{(2)}) \\
		&\stackrel{\mathmakebox[\widthof{   =   }]{}}{=} \alpha(v)\eta(k)\eta(S(h_{(1)}))\eta(h_{(2)})\eps(u) \\
		&\stackrel{\mathmakebox[\widthof{   =   }]{\ref{PR4}}}{=} \alpha(v)\eta(kS(h_{(1)}))\eta(h_{(2)})\eps(u) \\
		&\stackrel{\mathmakebox[\widthof{   =   }]{}}{=} \alpha(v)\eta(kS(h_{(1)}))\alpha(S(u_{(1)}))\alpha(u_{(2)})\eta(h_{(2)}) \\
		&\stackrel{\mathmakebox[\widthof{   =   }]{}}{=} \alpha(v)\alpha(S(u_{(1)})^R)\eta((kS(h_{(1)}))^R)\alpha(u_{(2)})\eta(h_{(2)}) \\
		&\stackrel{\mathmakebox[\widthof{   =   }]{\eqref{eq:leftmult}}}{=} \alpha(v)\alpha((S(u_{(1)})^{Rr})\eta(k^rS(h_{(1)})^R)\alpha(u_{(2)})\eta(h_{(2)}) \\
		&\stackrel{\mathmakebox[\widthof{   =   }]{}}{=} \alpha(v(S(u_{(1)})^{Rr})\eta(k^rS(h_{(1)})^R)\alpha(u_{(2)})\eta(h_{(2)}) \\
		&\stackrel{\mathmakebox[\widthof{   =   }]{}}{=} \pi_{\alpha, \eta}(v(S(u_{(1)})^{Rr}\# k^rS(h_{(1)})^R)\pi_{\alpha, \eta}(u_{(2)}\# h_{(2)}) \\
		&\stackrel{\mathmakebox[\widthof{   =   }]{}}{=} \pi_{\alpha, \eta}((v\# k)(S(u_{(1)})^R\# S(h_{(1)})^R))\pi_{\alpha, \eta}(u_{(2)}\# h_{(2)}).
	\end{align*}
	Hence $\pi_{\alpha, \eta}$ is a partial representation of $U\#_RH$. 
\end{proof}

Suppose now that \(U = U_{par},\) i.\,e.\ every partial representation of \(U\) is global. This happens for instance if \(U\) is connected, see \cref{cor:connected}. In this case, every partial representation of \(U \#_R H\) is in \(\PRep^{\mathrm{glob}}_{U \#_R H}\). Indeed, since \(\iota_U : U\to U \#_R H \) is a Hopf algebra morphism, every partial representation \(\pi\) of \(U\#_R H\) induces a partial representation \(\alpha : u \mapsto \pi(u \#_R 1_H)\) of \(U\). But this means that \(\alpha\) is an algebra morphism because \(U = U_{par}\), so \(\pi(uu' \#_R 1_H) = \pi(u \#_R 1_H) \pi(u' \#_R 1_H)\) for all \(u, u' \in U\). We can conclude the following.

\begin{theorem}
	\label{th:smashconnected}
	Let \(U, H\) be Hopf algebras and let \(R : H \otimes U \to U \otimes H\) be normal, multiplicative and a coalgebra map. Suppose that \(U = U_{par}\). Then there are isomorphisms of categories
	\[\mathcal{Q}_{U, H}^R \simeq \PRep_{U \#_R H} \simeq \Rep_{U \#_{\mathcal{T}} H_{par}}.\]
	In particular \((U \#_R H)_{par}\cong U \#_{\mathcal{T}} H_{par}\).
\end{theorem}
Explicitly, the isomorphism \((U \#_R H)_{par}\cong U \#_{\mathcal{T}} H_{par}\) is found by applying the first isomorphism of categories to the algebra map \(U \to U \#_{\mathcal{T}} H_{par} : u \mapsto u \#_{\mathcal{T}} [1_H]\) and the partial representation \(H \to U \#_{\mathcal{T}} H_{par} : h \mapsto 1_U \#_{\mathcal{T}} [h]\). This gives a partial representation \(U \#_R H \to U \#_{\mathcal{T}} H_{par},\) and an algebra map
\[(U \#_R H)_{par} \to U \#_{\mathcal{T}} H_{par} : [u \#_R h] \mapsto u \#_\mathcal{T} [h].\]
Its inverse is
\[U \#_{\mathcal{T}} H_{par} \to (U \#_R H)_{par} : u \#_{\mathcal{T}} [h_1] \cdots [h_n] \mapsto [u \#_R h_1] \cdots [1_U \#_R h_n].\]
Indeed, thanks to the fact that \([u \#_R 1_H][u' \#_R 1_H] = [uu' \#_R 1_H]\) because \(U = U_{par}\),
\[[u_1 u_2^{R_{21}} \cdots u_n^{R_{n, n - 1} \cdots R_{n1}} \#_R h_1^{R_{21} \cdots R_{n1}}][1_U \#_R h_2^{R_{32} \cdots R_{n2}}] \cdots [1_U \#_R h_n] = [u_1 \#_R h_1] \cdots [u_n \#_R h_n].\]

\begin{remark}
    Theorems \ref{th:isosP}, \ref{th:smashglobalpartial} and \ref{th:smashconnected} can also be formulated in terms of partial modules rather than partial representations, by considering only objects which are (pairs of) partial representations on an algebra of the form \(B = \End_k(V)\) for some vector space \(V\).
\end{remark}

Suppose again that \(R\) is invertible. Having described \((U \#_R H)_{par}\) in \cref{th:smashconnected}, let us take a look at the subalgebras \(A_{par}(U \#_R H)\) and \(\tilde{A}_{par}(U \#_R H)\). Recall that \(A_{par}(U \#_R H)\) is generated by the elements \({\varepsilon}_{u \#_R h}\) \eqref{eq:epsilonh}, and that \(A_{par}(U \#_R H)\) is the base algebra of the (left) bialgebroid structure on \(U \#_R H\). The other natural subalgebra \(\tilde{A}_{par}(U \#_R H)\), which is generated by the elements \(\tilde{\varepsilon}_{u \#_R h} = [S(u_{(1)} \#_R h_{(1)})][u_{(2)} \#_R h_{(2)}]\), is opposite to \(A_{par}(U \#_R H)\) and is the base algebra of the right bialgebroid structure on \((U \#_R H)_{par}\) (see \cite[\S 4.2]{alves2015partial}). We calculate
\begin{align*}
	\varepsilon_{u \#_R h} &\stackrel{\mathmakebox[\widthof{  =  }]{}}{=} (u_{(1)} \#_{\mathcal{T}} [h_{(1)}])(S(u_{(2)})^R \#_\mathcal{T} [S(h_{(2)})^R] ) \\
	&\stackrel{\mathmakebox[\widthof{  =  }]{}}{=} u_{(1)} S(u_{(2)})^{Rr} \#_{\mathcal{T}} [{h_{(1)}}^r][S(h_{(2)})^R] \\
	&\stackrel{\mathmakebox[\widthof{  =  }]{\eqref{eq:antipode}}}{=} {u_{(1)}}^{\rho'\rho} S({u_{(2)}}^{R'})^{r} \#_{\mathcal{T}} [{h_{(1)}}^{\rho'\rho r}][S({h_{(2)}}^{R'})] \\
	&\stackrel{\mathmakebox[\widthof{  =  }]{\eqref{eq:Rcoalg}}}{=} {{u^{R'}}_{(1)}}^{\rho} S({{u^{R'}}_{(2)}})^{r} \#_{\mathcal{T}} [{{h^{R'}}_{(1)}}^{\rho r}][S({h^{R'}}_{(2)})] \\
	&\stackrel{\mathmakebox[\widthof{  =  }]{\eqref{eq:rightmult}}}{=} ({{u^{R'}}_{(1)}} S({{u^{R'}}_{(2)}}))^r \#_{\mathcal{T}} [{{h^{R'}}_{(1)}}^r][S({h^{R'}}_{(2)})] \\
	&\stackrel{\mathmakebox[\widthof{  =  }]{}}{=} \epsilon(u^{R'}) 1_U \#_\mathcal{R} \varepsilon_{h^{R'}} \\
	&\stackrel{\mathmakebox[\widthof{  =  }]{}}{=} 1_U \#_{\mathcal{T}} \varepsilon_{u \blacktriangleright h}; 
 \end{align*}
 \begin{align*}
	\tilde{\varepsilon}_{u \#_R h} &\stackrel{\mathmakebox[\widthof{  =  }]{}}{=} (S(u_{(1)})^R \#_{\mathcal{T}} [S(h_{(1)})^R])(u_{(2)} \#_{\mathcal{T}} [S(h_{(2)})]) \\
	&\stackrel{\mathmakebox[\widthof{  =  }]{}}{=} S(u_{(1)})^R {u_{(2)}}^r \#_{\mathcal{T}} [S(h_{(1)})^{Rr}][h_{(2)}] \\
	&\stackrel{\mathmakebox[\widthof{  =  }]{\eqref{eq:rightmult}}}{=} (S(u_{(1)}) u_{(2)})^R \#_{\mathcal{T}} [S(h_{(1)})^R][h_{(2)}] \\
	&\stackrel{\mathmakebox[\widthof{  =  }]{}}{=} \epsilon(u) 1_U \#_{\mathcal{T}} \tilde{\varepsilon}_h.
\end{align*}

\begin{proposition}
	The assignments
	\begin{equation*}
		\alpha : \varepsilon_{u \#_R h} \mapsto \varepsilon_{u \blacktriangleright h}; \qquad \qquad
		\alpha' : \varepsilon_h \mapsto \varepsilon_{1 \#_R h}
	\end{equation*}
	provide an isomorphism of algebras \(A_{par}(U \#_R H) \cong A_{par}(H)\), and the assignments
	\begin{equation*}
		\beta : \tilde{\varepsilon}_{u \#_R h} \mapsto \epsilon(u) \tilde{\varepsilon}_{h}; \qquad \qquad
		 \beta' : \tilde{\varepsilon}_h \mapsto \tilde{\varepsilon}_{1 \#_R h}
	\end{equation*}
	provide an isomorphism of algebras \(\tilde{A}_{par}(U \#_R H) \cong \tilde{A}_{par}(H)\).
\end{proposition}
\begin{proof}
	The nontrivial part is to see that \(\alpha\) and \(\beta\) are well-defined. We treat \(\alpha\), the proof for \(\beta\) is analogous. We need to show that relations \eqref{eq:Arelation1}, \eqref{eq:Arelation2} and \eqref{eq:Arelation3} are preserved. The first is immediate, and the second follows directly from the fact that \(R\) is a coalgebra map. For the third, let us show that
	\begin{align*}
        \alpha(\varepsilon_{u_{(1)} \#_R h_{(1)}}\varepsilon_{(u_{(2)} \#_R h_{(2)})(v \#_R k)}) &= \varepsilon_{u_{(1)} \blacktriangleright h_{(1)}} \varepsilon_{u_{(2)}v^R \blacktriangleright {h_{(2)}}^R k} \\ &= \varepsilon_{u_{(1)} v^R \blacktriangleright {h_{(1)}}^R k} \varepsilon_{u_{(2)} \blacktriangleright h_{(2)}} \\ &= \alpha(\varepsilon_{(u_{(1)} \#_R h_{(1)})(v \#_R k)} \varepsilon_{u_{(2)} \#_R h_{(1)}}).
	\end{align*}
	Indeed,
 \begingroup
    \allowdisplaybreaks
	\begin{align*}
		\varepsilon_{u_{(1)} \blacktriangleright h_{(1)}} \varepsilon_{u_{(2)}v^R \blacktriangleright {h_{(2)}}^R k} &\stackrel{\mathmakebox[\widthof{  =  }]{}}{=} \epsilon({u_{(1)}}^{\rho'}) \epsilon((u_{(2)}v^R)^{R'}) \varepsilon_{{h_{(1)}}^{\rho'}} \varepsilon_{({h_{(2)}}^Rk)^{R'}} \\
		&\stackrel{\mathmakebox[\widthof{  =  }]{\eqref{eq:leftmult}}}{=} \epsilon({u_{(1)}}^{\rho'}) \epsilon((u_{(2)}v^R)^{R'_1R'_2}) \varepsilon_{{h_{(1)}}^{\rho'}} \varepsilon_{{h_{(2)}}^{RR'_1}k^{R'_2}} \\
		&\stackrel{\mathmakebox[\widthof{  =  }]{\eqref{eq:rightmult}}}{=} \epsilon({u_{(1)}}^{\rho'}) \epsilon({u_{(2)}}^{r'_1 r'_2}v^{RR'_1 R'_2}) \varepsilon_{{h_{(1)}}^{\rho'}} \varepsilon_{{h_{(2)}}^{RR'_1 r'_1}k^{R'_2 r'_2}} \\
		&\stackrel{\mathmakebox[\widthof{  =  }]{}}{=} \epsilon({u_{(1)}}^{\rho'}) \epsilon({u_{(2)}}^{r'_1 r'_2}) \epsilon(v^{R'}) \varepsilon_{{h_{(1)}}^{\rho'}} \varepsilon_{{h_{(2)}}^{r'_1}k^{R'r'_2}} \\
		&\stackrel{\mathmakebox[\widthof{  =  }]{\eqref{eq:Rcoalg}}}{=} \epsilon({u^{\rho'}}_{(1)}) \epsilon({{u^{\rho'}}_{(2)}}^{r'}) \epsilon(v^{R'}) \varepsilon_{{h^{\rho'}}_{(1)}} \varepsilon_{{h^{\rho'}}_{(2)}k^{R'r'}} \\
		&\stackrel{\mathmakebox[\widthof{  =  }]{\eqref{eq:Arelation3}}}{=} \epsilon(u^{\rho'r'}) \epsilon(v^{R'}) \varepsilon_{{{h^{\rho'}}_{(1)}} k^{R' r'}} \varepsilon_{{h^{\rho'}}_{(2)}},
  \end{align*}
  \endgroup
  and
  \begin{align*}
		\varepsilon_{u_{(1)} v^R \blacktriangleright {h_{(1)}}^R k} \varepsilon_{u_{(2)} \blacktriangleright h_{(2)}} &\stackrel{\mathmakebox[\widthof{  =  }]{}}{=} \epsilon((a_{(1)}v^R)^{R'}) \epsilon({u_{(2)}}^{\rho'}) \varepsilon_{({h_{(1)}}^R k)^{R'}} \varepsilon_{{h_{(2)}}^{\rho'}} \\
		&\stackrel{\mathmakebox[\widthof{  =  }]{\eqref{eq:rightmult}}}{=} \epsilon({u_{(1)}}^{r'} v^{RR'}) \epsilon({u_{(2)}}^{\rho'}) \varepsilon_{({h_{(1)}}^R k)^{R'r'}} \varepsilon_{{h_{(2)}}^{\rho'}} \\
		&\stackrel{\mathmakebox[\widthof{  =  }]{\eqref{eq:leftmult}}}{=} \epsilon({u_{(1)}}^{r'_1 r'_2} v^{RR'_1 R'_2}) \epsilon({u_{(2)}}^{\rho'}) \varepsilon_{{h_{(1)}}^{RR'_1 r'_1} k^{R'_2 r'_2}} \varepsilon_{{h_{(2)}}^{\rho'}} \\
		&\stackrel{\mathmakebox[\widthof{  =  }]{\eqref{eq:Rcoalg}}}{=} \epsilon({{u^{\rho'}}_{(1)}}^{r'} v^{R'}) \epsilon({u^{\rho'}}_{(2)}) \varepsilon_{{{h^{\rho'}}_{(1)}} k^{R' r'}} \varepsilon_{{h^{\rho'}}_{(2)}} \\
		&\stackrel{\mathmakebox[\widthof{  =  }]{}}{=} \epsilon(u^{\rho'r'}) \epsilon(v^{R'}) \varepsilon_{{{h^{\rho'}}_{(1)}} k^{R' r'}} \varepsilon_{{h^{\rho'}}_{(2)}}. \qedhere
	\end{align*}
\end{proof}

\subsection{Partial representations of cocommutative Hopf algebras}

Suppose that \(k\) has characteristic 0 and let \(H\) be a cocommutative Hopf algebra. Let \(G\) be the group of grouplikes of \(H\), and \(\mathfrak{g}\) the vector space of its primitive elements. Then \(\mathfrak{g}\) is a Lie algebra, and the subalgebra generated by \(\mathfrak{g}\) in \(H\) is isomorphic to \(U(\mathfrak{g})\), the universal enveloping algebra of \(\mathfrak{g}\). The Cartier-Gabriel-Konstant-Milnor-Moore theorem states that
\[H \cong U(\mathfrak{g}) \#_R kG,\]
where \(R : kG \otimes U(\mathfrak{g}) \to U(\mathfrak{g})\) is given on generators by
\[R(g \otimes x) = gxg^{-1} \otimes g.\]
Since \(U(\mathfrak{g})\) is a connected Hopf algebra, we can apply \cref{th:smashconnected} and conclude that
\[H_{par} \cong U(\mathfrak{g}) \#_{\mathcal{T}} k_{par} G,\]
where 
\begin{align*}
    \mathcal{T} : k_{par} G \otimes U(\mathfrak{g}) &\to U(\mathfrak{g}) \otimes k_{par} G, \\ [g_1] \cdots [g_n] \otimes x &\mapsto g_1 \cdots g_n x g_n^{-1} \cdots g_1^{-1} \otimes [g_1] \cdots [g_n].
\end{align*}
Suppose now that \(H\) has only finitely many grouplikes, so that \(G\) is a finite group. Then by \cite[Corollary 2.7]{DEP},  \(k_{par} G\) is isomorphic to the groupoid algebra \(k\Gamma(G)\), where 
\[\Gamma(G) = \{(A, g) \in \mathcal{P}(G) \times G \mid  1, g^{-1} \in A\}\]
with multiplication \((A, g)(B, h) = (B, gh)\) whenever \(A = hB\) and not defined else. Explicitly, the isomorphism is given by
\begin{align}
	k\Gamma(G) &\longleftrightarrow  k_{par} G \nonumber \\ 
	(A, g) &\longmapsto	[g] \prod_{h \in A} [h][h^{-1}] \prod_{h \notin A} (1 - [h][h^{-1}]), \label{eq:isogroupoid}\\
	\sum_{A \ni 1, g^{-1}} (A, g) &\longmapsfrom [g]. \nonumber
\end{align}
Using this, we can describe \(H_{par}\) as a weak Hopf algebra coming from a Hopf category \(\mathcal{H}\). We refer to \cite{weakHopf} and \cite{HopfCat} for the definitions of a weak Hopf algebra and a Hopf category. We define a Hopf category \(\mathcal{H}\) with objects the subsets \(A\) of \(G\) containing 1 (as for the groupoid \(\Gamma(G)\)). For two of these subsets \(A\) and \(B\), put
\begin{gather*}
	K_{A, B} = \{g \in G \mid gA = B\}, \\
	\mathsf{Hom}_{\mathcal{H}}(A, B) = U(\mathfrak{g}) \otimes kK_{A, B}.
\end{gather*}
Then \(\mathsf{Hom}_{\mathcal{H}}(A, B)\) is a coalgebra, and the composition maps
\begin{equation}
    \label{eq:comp}
    \mathsf{Hom}_{\mathcal{H}}(B, C) \otimes \mathsf{Hom}_{\mathcal{H}}(A, B) \to \mathsf{Hom}_{\mathcal{H}}(A, C) : (x \otimes g) \circ (x' \otimes g') = (xgx'g^{-1} \otimes gg')
\end{equation}
are coalgebra maps. The antipode on \(\mathcal{H}\) is given by
\[S_{A, B} : \mathsf{Hom}_{\mathcal{H}}(A, B) \to \mathsf{Hom}_{\mathcal{H}}(B, A) : x \otimes g \mapsto S(x) \otimes g^{-1}.\]
The weak Hopf algebra associated to \(\mathcal{H}\) is
\begin{equation}
	\label{eq:weakHopf}
	W = \bigoplus_{1 \in A, B \subseteq G} \mathsf{Hom}_{\mathcal{H}}(A, B) = U(\mathfrak{g}) \otimes \bigoplus_{1 \in A, B \subseteq G} K_{A, B}.
\end{equation}
Note however that the algebra structure on the right hand side is not the tensor product algebra structure, but is obtained from the composition rule \eqref{eq:comp}. This way \(W\) is the smash product of \(U(\mathfrak{g})\) and \(\bigoplus_{1 \in A, B \subseteq G} K_{A, B}\). This last space can be identified with the groupoid algebra \(k\Gamma(G),\) so we can combine \eqref{eq:isogroupoid} with \eqref{eq:weakHopf} and obtain an isomorphism of weak Hopf algebras \(H_{par} \cong W\). We remark that this weak Hopf algebra is cocommutative, and when considered as a Hopf algebroid, its base algebra is commutative.

\subsection{Further examples}
\subsubsection{Exact factorizations of groups}
Let \(G\) be a group with unit \(e\) and \(G = ML\) an \textit{exact factorization}, i.\,e.\ \(M\) and \(L\) are subgroups of \(G\) such that every element of \(G\) can be written in a unique way as a product of an element of \(M\) and an element of \(L\). Note that also \(G = G^{-1} = L^{-1} M^{-1} = LM\) is an exact factorization. 

The exact factorization induces a local braiding \(R : kL \otimes kM \to kM \otimes kL : l \otimes m \mapsto m' \otimes l'\) where \(lm = m'l'\). It is easy to check that \(R\) is normal, multiplicative and a coalgebra map, so we get an \(R\)-smash product \(kM \#_R kL\) which is isomorphic to \(kG\) as a Hopf algebra, via the multiplication map.

By \cref{th:smashglobalpartial}, partial representations of \(kG\) that are global on the subgroup \(kM\) are equivalent to representations of the algebra \(kM \#_{\mathcal{T}} k_{par}L\). Suppose now that \(L\) is finite. In that case  \(k_{par} L\) is isomorphic to the groupoid algebra \(k\Gamma(L)\) as in \eqref{eq:isogroupoid}.

In \cite{DHSV}, partial representations of \(G\) which are global on a subgroup \(M\) were studied. It was shown that they correspond to representations of the groupoid algebra \(k\Gamma_{M}(G)\), where \(\Gamma_M(G)\) is the groupoid
\[\Gamma_M(G) = \{(A, g) \in \mathcal{P}(G/M) \times G \mid M, g^{-1}M \in A\}.\] 
In fact, \(k\Gamma_{M}(G)\) is isomorphic to \(kM \#_{\mathcal{T}} k_{par}L\). 
The exact factorization \(G = LM\) induces a bijection \(\varphi : L \to G/M : l \mapsto lM\). One can check that
\[\theta : kM \#_{\mathcal{T}} k_{par}L \to k\Gamma_M(L) : m \#_{\mathcal{T}} [l] \mapsto \sum_{A \ni e, l^{-1}} (\varphi(A), ml) \]
defines an isomorphism of algebras.

\subsubsection{Graded partial modules}

Let \(G, F\) be finite groups and \(\triangleright : F \times G \to G\) an action of \(F\) on \(G\) by automorphisms. Then \(kF\) acts on \(kG^*\) via
\[f \cdot p_g = p_{f \triangleright g},\]
where \(p_g\) is the dual basis vector to \(g \in G\).
This induces a twist map (cf.\ \cref{ex:hopfactionR})
\[R : kF \otimes kG^* \to kG^* \otimes kF : f \otimes p_g \mapsto p_{f \triangleright g} \otimes f.\]
This map is normal, multiplicative and a coalgebra map, so the smash product \(kG^* \#_R kF\) is a Hopf algebra. Modules over this Hopf algebra are equivalent to \(G\)-graded \(F\)-representations, i.\,e.\ \(G\)-graded vector spaces \(V = \bigoplus_{g \in G} V_g\) together with a group morphism \(\rho : F \to \mathsf{GL}(V)\) such that \(\rho(f)(V_g) \subseteq V_{f \triangleright g}\) for all \(f \in F, g \in G\). Indeed, a \(G\)-graded vector space is a \(kG\)-comodule, which is equivalent to a \(kG^*\)-module. If \(v = \sum_{g \in G} v_g,\) where \(v_g \in V_g\) for each \(g \in G,\) then the \(kG^*\)-module structure is given by \(p_g \cdot v = v_g\).

The notion of graded module can be extended to the partial world in the following way:
\begin{definition}
	A \(G\)-graded partial \(F\)-representation is a vector space \(V = \bigoplus_{g \in G} V_g\) together with a partial representation \(\eta : F \to \mathsf{End}_k(V)\) such that \(\eta(f)(V_g) \subseteq V_{f \triangleright g}\) for all \(f \in F, g \in G\).
\end{definition}

As in the global case, the \(G\)-grading on \(V\) induces a \(kG^*\)-module structure, in other words, a representation
\[\alpha : kG^* \to \End_k(V) : p_g \mapsto (v \mapsto v_g).\]
Now \(\eta(f)\alpha(p_g) (v) = \eta(f) (v_g) \in V_{f \triangleright g}\). So for any \(h \in G,\) \(\alpha(p_{f \triangleright g}) \eta(f)(v_h) = \delta_{gh} \eta(f)(v_h)\), because \(f \triangleright g \neq f \triangleright h\) if \(g \neq h\).  This implies that
\[\eta(f) \alpha(p_g)(v) =  \eta(f) (v_g) = \alpha(p_{f \triangleright g}) \eta(f)(v_g) = \sum_{h \in G} \alpha(p_{f \triangleright g}) \eta(f)(v_h) = \alpha(p_{f \triangleright g}) \eta(f)(v).\]
Hence \(\eta(f) \alpha(p_g) = \alpha({p_g}^R) \eta(f^R)\) for all \(f \in F, g \in G,\) so \cref{th:smashglobalpartial} can be applied to obtain the following corollary.

\begin{corollary}
	Let \(F\) and \(G\) be finite groups and let \(\triangleright\) be a left action of \(F\) on \(G\) by automorphisms. Then the category of \(G\)-graded partial \(F\)-modules is equivalent to the category of left \(kG^* \#_{\mathcal{T}} k_{par} F\)-modules, where 
	\[\mathcal{T} : k_{par} F \otimes kG^* \to kG^* \otimes k_{par} F : [f_1] \cdots [f_n] \otimes p_g \mapsto p_{(f_1 \cdots f_n) \triangleright g} \otimes [f_1] \cdots [f_n].\]
\end{corollary}

Suppose now that \(G\) is a \(p\)-group and that \(k\) has characteristic \(p\). Then we saw in \cref{ex:connected} that \(kG^*\) is a connected Hopf algebra, so by \cref{cor:connected}, \((kG^*)_{par} = kG^*\). From \cref{th:smashconnected} it follows that in this case
\[(kG^* \#_R kF)_{par} = kG^* \#_{\mathcal{T}} k_{par} F.\]
For more concrete examples, we refer to \cite{tiagothesis}.
	
	\subsubsection{Drinfel'd double}

 Let \(H\) be a finite-dimensional Hopf algebra. As is done in \cite[Example 7.2.5]{majid}, the Drinfel'd double (or quantum double) \(D(H)\) can be described as the smash product \((H^*)^{op} \#_R H,\) where
 \[R : H \otimes H^* \to H^* \otimes H : h \otimes \psi \mapsto \psi_{(1)}(S(h_{(1)})) \psi_{(3)}(h_{(3)})\ \psi_{(2)} \otimes h_{(2)}.\]
 This map induces a normal and multiplicative twist map
 \[\mathcal{R} : H_{par} \otimes (H^*)^{op}_{par} \to (H^*)^{op}_{par} \to H_{par}\]
 by \cref{le:mathcalR}. By \cref{th:isosP}, a right module over \((H^*)^{op}_{par} \#_{\mathcal{R}} H_{par}\) can be interpreted as a right partial \(H\)-module, which is at the same time a right partial \((H^*)^{op}\)-module (i.\,e.\ a left partial \(H^*\)-module) compatible through \(R\). By \cite[Theorem 4.14]{ABCQVparcorep}, left partial \(H^*\)-modules are equivalent to right partial \(H\)-comodules. Recall also that right \(D(H)\)-modules are equivalent to right-right Yetter-Drinfel'd modules over \(H\). So a right \((H^*)^{op}_{par} \#_{\mathcal{R}} H_{par}\)-module is a right partial \(H\)-module \(M\), which is at the same time a right partial \(H\)-comodule satisfying the Yetter-Drinfel'd compatibility condition
 \[(m \cdot h)^{[0]} \otimes (m \cdot h)^{[1]} = m^{[0]} \cdot h_{(2)} \otimes S(h_{(1)}) m^{[1]} h_{(3)}.\]
 for all \(h \in H, m \in M\). This is exactly \eqref{eq:split2} interpreted in terms of the partial action and coaction of \(H\) on \(M\).

\subsection*{Acknowledgements} The authors thank Marcelo Alves and Joost Vercruysse for the interesting discussions, useful remarks and great advice. WH would like to thank TF and AN for the warm welcome during his visit to Curitiba in February 2024.

\end{document}